\documentclass[12pts]{amsart}
\usepackage[bottom=3.5cm,top=3.5cm,left=3.5cm,right=3.5cm]  {geometry}           
\usepackage{graphicx}
\usepackage[utf8]{inputenc}
\usepackage{amssymb}
\usepackage{dsfont}
\usepackage{amsfonts}
\usepackage{tikz}
\usepackage{enumitem}

\usetikzlibrary{patterns,arrows,decorations.pathreplacing}
\allowdisplaybreaks
\newtheorem{theorem}{Theorem}[section]
\newtheorem{coro}[theorem]{Corollary}
\newtheorem{proposition}[theorem]{Proposition}
\newtheorem{definition}[theorem]{Definition}
\newtheorem{lemma}[theorem]{Lemma}

\numberwithin{equation}{section}

\newcommand{\N}{\mathbb{N}}
\newcommand{\Z}{\mathbb{Z}}

\newcommand{\R}{\mathbb{R}}

\newcommand{\RR}{\mathcal{R}}
\newcommand{\FF}{\mathcal{F}}
\newcommand{\EE}{\mathcal{E}}
\newcommand{\CC}{\mathcal{C}}
\newcommand{\BB}{\mathcal{B}}
\newcommand{\PP}{\mathcal{P}}
\newcommand{\LL}{\mathcal{L}}
\newcommand{\HH}{\mathcal{H}}
\newcommand{\NN}{\mathcal{N}}
\newcommand{\m}{\mathfrak{m}}
\renewcommand{\SS}{\mathcal{S}}
\newcommand{\s}{\textnormal{\textsf{s}}}
\newcommand{\dd}{\mathrm{d}}

\begin{document}
\title{Products of snowflaked euclidean lines are not minimal for looking down }
\author{Matthieu Joseph}
\address{D\'epartement de Math\'ematiques\\ \'Ecole Normale Sup\'erieure de Lyon\\ 69364 Lyon Cedex 07, France}
\email{matthieu.joseph@ens-lyon.fr}
\thanks{M.J. is supported by Erasmus and ExploRA'Sup grants}
\author{Tapio Rajala}
\address{University of Jyvaskyla, Department of Mathematics and Statistics,
P.O. Box 35 (MaD), FI-40014 University of Jyvaskyla, Finland}
\email{tapio.m.rajala@jyu.fi}
\thanks{T.R. is supported by the Academy of Finland project no. 274372.}
\date{\today}
\maketitle
\begin{abstract}
We show that products of snowflaked Euclidean lines are not minimal for looking down. This question was raised in \textit{Fractured fractals and broken dreams}, Problem $11.17$, by David and Semmes.

The proof uses arguments developed by Le Donne, Li and Rajala to prove that the Heisenberg group is not minimal for looking down. By a method of shortcuts, we define a new distance $d$ such that the product of snowflaked Euclidean lines looks down on $(\R^N,d)$, but not vice versa.
\end{abstract}

\tableofcontents

\section{Introduction}

The concept of BPI space (Big Pieces of Itself) was introduced by David and Semmes in \cite{david} in order to provide a framework in which to work with self-similarity in metric spaces setting. A BPI space is more or less a metric space in which any two balls contain big pieces that look almost the same up to scaling and bounded distortions. They also introduced a notion of BPI equivalence in order to understand and classify BPI geometries. Two BPI spaces are BPI equivalent if they possess pieces of positive measure that are biLipschitz equivalent. With the aim of classifying BPI spaces that are not BPI equivalent, they defined a notion of looking down between BPI spaces of the same dimension. A natural question arises when working with looking down BPI spaces: what are the most primitive BPI spaces? Such BPI spaces are called minimal for looking down (see Section \ref{preli} for the definitions). 

By using ideas of \cite{tapio}, where they proved that the Heisenberg group is not minimal for looking down, we prove the following theorem, which gives an answer to Problem $11.17$ in \cite{david}. 
\begin{theorem}\label{theorem} Given $s_1,\dots s_N\in(0,1]$, the space $\displaystyle\Big( \R^N,\sum_{k=1}^N|x_k-y_k|^{s_k}\Big)$
is minimal for looking down if and only if $\displaystyle\sum_{1\leq k\leq N} s_k =N$. 
\end{theorem}
\noindent
The distance defined above will be denoted by $d_\textsf{s}$, where $\textsf{s}$ stands for $(s_1,\dots,s_N)\in (0,1]^N$. It is the $\ell^1$ product distance of the snowflaked distances $|\cdot|^{s_k}$ defined on $\R$: if $x=(x_1,\dots,x_N)$ and $y=(y_1,\dots,y_N)$ are in $\R^N$, then 
$$d_{\textsf{s}}(x,y)=\sum_{k=1}^N|x_k-y_k|^{s_k}.$$

Kirchheim proved in \cite{kirchheim} that Euclidean spaces are minimal for looking down, that is, if $\s=(1,\dots, 1)$, then $(\R^N,d_\s)$ is minimal for looking down. To prove Theorem \ref{theorem}, it is thus sufficient to show that if $\sum s_k <N$, then the space $(\R^N,d_\textsf{s})$ is not minimal for looking down. \\

From now on we fix an integer $N\geq 1$ and an $N$-tuple ${\s=(s_1,\dots,s_N)\in (0,1]^N}$ such that $\sum s_k<N$. We denote by $s$ the \textit{minimum snowflaking factor}, i.e. $s=\min\left\{s_k,\ 1\leq k\leq N\right\}$, and by $L$ the \textit{minimally snowflaked layer}, that is the subset of $\left\{1,\dots,N\right\}$ where the snowflaking factor is minimum:
$$L=\left\{k\in\left\{1,\dots,N\right\},\ s_k=s\right\}.$$

The strategy to prove Theorem \ref{theorem} is the following. First we look at a one dimensional problem. We construct a quotient semi-distance $d_\RR$ on $\R$ associated with an equivalence relation $\RR$ using the shortening technique developed in \cite{tapio}. The construction of $d_\RR$ is made in a self-similar way so that the subspace $[0,1]$ of $\R$, endowed with this semi-distance $d_\RR$ is a BPI space. To be more precise, the quotient space $([0,1]/d_\RR,\overline{d}_\RR)$ is a BPI space. The semi-distance $d_\RR$ verifies $d_\RR\leq |\cdot|^s$, where $|\cdot|$ is the Euclidean distance on $\R$, and $s$ the minimum snowflaking factor. Moreover any Lipschitz function from $([0,1],d_{\RR})$ to $(\R,|\cdot|^s)$ is constant.

Then, we look at the $N$-dimensional problem. We slightly modify the distance $d_\textsf{s}$ in the minimally snowflaked layer $L$ by replacing the terms $|\cdot|^{s}$ with the semi-distance $d_\RR$. This gives a new semi-distance $d_{\s,\RR}$ on $\R^N$. As a product of bounded BPI spaces is a BPI space, the quotient space $([0,1]^N/d_{\s,\RR},\overline{d}_{\s,\RR})$ is a BPI space. Suppose then that $([0,1]^N,d_{\s,\RR})$ looks down on $(\R^N,d_\s)$. There exists a Lipschitz map $g:(A,d_{\s,\RR})\to (\R^N,d_\s)$ such that $\HH^\alpha\big(g(A)\big)>0$, where $\alpha$ is the Ahlfors dimension of $(\R^N,d_\s)$ and $A$ a closed subset of $[0,1]^N$. By a blow-up argument, we prove that there exists a Lipschitz map $f:([0,1]^N,d_{\s,\RR})\to (\R^N,d_{\s})$ whose image has positive measure, which is in contradiction with the property on Lipschitz functions from $([0,1],d_{\RR})$ to $(\R,|\cdot|^s)$.

Section \ref{preli} deals with definitions related to BPI spaces, quotient semi-distance, etc. In Section \ref{product_BPI_section}, we prove that the product of two BPI spaces, both bounded or both unbounded, is a BPI space. In Section \ref{new semi-distance}, we construct the semi-distance $d_\RR$ on $\R$ and we prove that the metric space $([0,1],d_\RR)$ is a BPI space. In Section \ref{lipschitz_function}, we prove that every Lipschitz function from $([0,1],d_\RR)$ to $(\R,|\cdot|^s)$ is constant. Finally, in Sections \ref{blow_up} and \ref{conclusion}, we conclude by a blow-up process that the space $(\R^N,d_\s)$ is not minimal for looking down.

\section{Preliminaries}\label{preli}

In what follows, $\N=\{1,2,3,\dots\}$.
By a measure $\m$ on a metric space $(X,d)$ we always mean an outer measure such that Borel sets are $\m$-measurable. Recall that an outer measure $\m$ on a set $X$ is a map $\m : \PP(X)\to [0,+\infty]$ defined on all subsets of $X$, such that $\m(\varnothing)=0$, $\m(A)\leq \m(B)$ for all $A,B$ subsets of $X$ with $A\subset B$, and for all countable sequences $(A_n)_{n\in\N}$ of subsets of $X$, 
$$\m\left(\bigcup_{n=1}^{+\infty}A_n\right)\leq \sum_{n=1}^{+\infty}\m(A_n).$$

Any metric space $(X,d)$ can be endowed with a one-parameter family of natural measures: for all $\alpha>0$, we define the $\alpha$-dimensional Hausdorff measure $\HH^\alpha_d$ (or just $\HH^\alpha$ when the distance is implicit) as follows: for all $A\subset X$, 
$$\HH^\alpha_d(A)=\underset{\delta >0}{\lim}\ \inf\left\{\sum_{n=1}^{+\infty}\big(\text{diam}_d (A_n)\big)^{\alpha},\ A\subset \bigcup_{n=1}^{+\infty}A_n,\ \text{diam}_d(A_n)\leq \delta\right\}.$$

\begin{definition}[Ahlfors regularity]\label{def:ahlfors}Let $\m$ be a measure on a complete metric space $(X,d)$, and $\alpha>0$. We say that the metric measure space $(X,d,\m)$ is Ahlfors regular of dimension $\alpha$ (or Ahlfors $\alpha$-regular) if there exists a constant $K>0$ such that for all $x\in X$ and $r\in(0,\textnormal{diam}(X)]$, 
$$K^{-1}r^\alpha\leq \m\big(B(x,r)\big)\leq Kr^\alpha.$$  
\end{definition}

In Definition \ref{def:ahlfors} and later on we follow the convention of \cite{david} where each ball $B(x,r)$ is implicitly assumed to have finite radius even if the range of radii would permit $r = \infty$.
The following well-known lemma (see \cite{david}, Lemma $1.2$) allows us to talk about Ahlfors regularity on a metric space $(X,d)$.

\begin{lemma}\label{sufficient_condition} If $(X,d,\m)$ is Ahlfors regular of dimension $\alpha$, then so is $(X,d,\HH^\alpha)$, and there exists a constant $K>0$ such that for all Borel sets $B\subset X$, $K^{-1}\m(B)\leq \HH^\alpha(B)\leq K\m(B)$.
\end{lemma}
It is well-known that closed and bounded sets of an Ahlfors regular space are compact, because closed subsets are totally bounded, and Ahlfors regular spaces are assumed to be complete.

Recall that a $C$-biLipschitz map $f:(X,d) \to (Y,\rho)$ between two metric spaces is a map $f:X\to Y$ such that for all $x,y\in X$, 
$$C^{-1}d(x,y)\leq \rho\big(f(x),f(y)\big)\leq Cd(x,y).$$

\begin{definition}[BPI space]
 Let $(X,d)$ be an Ahlfors $\alpha$-regular metric space. $(X,d)$ is a BPI space of dimension $\alpha$ if there exist $C,\theta >0$ such that for all $x,y\in X$ and all $r,t \in (0,\textnormal{diam}_dX]$, there exist a closed subset $A\subset B(x,r)$ with $\HH^\alpha(A)\geq \theta r^\alpha$ and a $C$-biLipschitz map $f:(A,r^{-1}d)\to (B(y,t),t^{-1}d)$.
\end{definition} 

We next define an equivalence relation for BPI spaces.

\begin{definition}[BPI equivalence] Two BPI spaces $(X,d)$ and $(Y,\rho)$ of the same dimension $\alpha$ are BPI equivalent if there exist $C,\theta >0$ such that for all $x\in X, y\in Y$ and all $r\in (0,\textnormal{diam}_d(X)]$, $t\in (0,\textnormal{diam}_\rho(Y)]$, there exist a closed subset $A\subset B_d(x,r)$ with $\HH^\alpha_d(A)\geq \theta r^\alpha$ and a $C$-biLipschitz map $f:(A,r^{-1}d)\to \big(B_\rho(y,t),t^{-1}\rho\big)$.
\end{definition}

This is an equivalent relation (see \cite{david}, Chapter $7$). The following definition allows us to compare BPI spaces of the same dimension that are not BPI equivalent. 

\begin{definition}[Looking down] Let $(X,d)$ and $(Y,\rho)$ be two BPI spaces of the same dimension $\alpha$. We say that $X$ looks down on $Y$ if there exist a closed subset $A\subset X$ and a Lipschitz map $g:A\to Y$ such that $\HH^\alpha_\rho\big(g(A)\big)>0$. 
\end{definition}

This is a partial order on the set of equivalence classes of BPI spaces with the equivalence relation "looking down equivalence" (two BPI spaces are \textit{looking down equivalent} if each looks down on the other, see \cite{david}, Chapter $11$). We propose here a slightly different definition of minimal for looking down from the one given in \cite{david} by David and Semmes. 
\begin{definition}[Minimal for looking down]\label{minimal} A BPI space $X$ is minimal for looking down if for any BPI space $Y$ such that $X$ looks down on $Y$, then $Y$ looks down on $X$.
\end{definition}
The original definition of David and Semmes says that a BPI space $X$ is minimal for looking down if for any BPI space Y such that $X$ looks down on $Y$, then $X$ and $Y$ are BPI equivalent. Definition \ref{minimal} is thus weaker, but more natural, since with this definition, a BPI space that is minimal for looking down is a BPI space minimal for the partial order "looking down". 
\\

Next, following Definition $3.1.12$ in \cite{burago}, we define the notion of quotient semi-distance, which is useful in the shortening technique used in Section \ref{new semi-distance}. Given an equivalence relation $\RR$ on a metric space $(X,d)$ we can construct the \textit{quotient semi-distance} $d_\mathcal{R}$ defined on $X$ by
$$d_\mathcal{R}(x,y)=\inf\Big\{\sum_{k=0}^nd(x_k,y_k),\ x\RR x_0,\ y_k\RR x_{k+1},\ y_n\RR y\Big\}.$$
The quotient semi-distance $d_\RR$ is a \textit{semi-distance}, that is $d_\RR$ is nonnegative, symmetric, verifies the triangle inequality, is zero on the diagonal of $X\times X$ but can be zero also outside the diagonal. 

A set $(x_0,y_0,\dots ,x_n,y_n)$ such that $x\mathcal{R}x_0,\ y_k\mathcal{R}x_{k+1}$ and $y_n\mathcal{R}y$ is called an \textit{itinerary} between $x$ and $y$. That is, one is allowed to take a shortcut by teleporting itself between $x$ and $x_0$, between $y_k$ and $x_{k+1}$ for all $k$ and between $y_n$ and $y$. An itinerary $(x_0,y_0,\dots, x_n,y_n)$ is \textit{shorter} than another itinerary $(x_0',y_0',\dots, x_n',y_n')$ if 
$$\sum_{k=0}^nd(x_k,y_k)\leq \sum_{k=0}^nd(x_k',y_k').$$

The next lemma gives a way to construct Lipschitz maps and similitudes for quotient semi-distances.

\begin{lemma}\label{lipschitz_map} Let $(X,d)$, $(Y,\rho)$ be two metric spaces, $\RR_X$ an equivalence relation on $X$ and $\RR_Y$ an equivalence relation on $Y$. Suppose that there exist a map $f:X\rightarrow Y$ and $\lambda >0$ such that for all $x,y\in X$, $\rho\big(f(x),f(y)\big)=\lambda d(x,y)$ and $x\RR_X y\Rightarrow f(x)\RR_Y f(y)$. Then for all $x,y\in X$, 
$$\rho_{\RR_Y}\big(f(x),f(y)\big)\leq\lambda d_{\RR_X}(x,y).$$ 
Moreover, if $f$ is bijective and $x\RR_X y\Leftrightarrow f(x)\RR_Y f(y)$, then $\rho_{\RR_Y}\big(f(x),f(y)\big)=\lambda d_{\RR_X}(x,y)$.
\end{lemma}

\begin{proof} Let $\varepsilon >0$ and $(x_0,y_0,\dots,x_n,y_n)$ be an itinerary from $x$ to $y$ such that 
$$\lambda^{-1}\sum_{k=0}^n \rho\big(f(x_k),f(y_k)\big)=\sum_{k=0}^nd(x_k,y_k)\leq d_\mathcal{R}(x,y)+\varepsilon.$$
Since, $\big(f(x_0),f(y_0),\dots,f(x_n),f(y_n)\big)$ is an itinerary from $f(x)$ to $f(y)$, then $\rho_{\RR_Y}\big(f(x),f(y)\big) \leq \lambda (d_{\RR_X}(x,y)+\varepsilon)$, and the result holds for $\varepsilon \to 0$. 
If $f$ is bijective and $x\RR_X y\Leftrightarrow f(x)\RR_Y f(y)$, it is sufficient to apply the foregoing to $f^{-1}$.\end{proof}

If $d$ is a semi-distance on a space $X$, we denote by $(X/d, \overline{d})$ the quotient metric space, which is the space of all equivalence classes for the relation $x\sim y\Leftrightarrow d(x,y)=0$. The natural distance $\overline{d}$ on $X/d$ is defined by $\overline{d}\big(\pi(x),\pi(y)\big)=d(x,y)$, where $\pi : X\to X/d$ is the canonical projection. One can easily check that $\overline{d}$ is well defined and is a distance on $X/d$.\\

Finally, we recall some basic facts about the Hausdorff distance. If $(X,d)$ is a metric space, we denote by $\CC(X)$ the set of all compact subsets of $X$. The $\varepsilon$-neighborhood of a set $A\subset X$, denoted by $A^\varepsilon$, is $\left\{x\in X, d(x,A)<\varepsilon\right\}.$ On $\CC(X)$, we consider the Hausdorff distance $d_H$ defined for $A,B\in \CC(X)$ by 
$$d_H(A,B)=\inf\left\{\varepsilon >0, A\subset B^\varepsilon, B\subset A^\varepsilon\right\}.$$
The space $(\CC(X),d_H)$ is a metric space, which is compact if $X$ is compact (Blaschke Theorem). 
We sometimes write $K_n\overset{H}{\longrightarrow}K$ to say that $(K_n)_{n\in\N}$ converges to $K$ in the Hausdorff distance. 
Moreover, if $(X,d)$ is Ahlfors $\alpha$-regular, then the $\alpha$-dimensional Hausdorff measure is upper semi-continuous on $(\CC(X),d_H)$.

\begin{lemma}\label{semi_continuity}
Let $(X,d)$ be an Ahlfors regular metric space of dimension $\alpha$. Let $(K_n)_{n\in\N}$ be a sequence of compact sets that converges for $d_H$ to $K\in\CC(X)$. Then
$$\underset{n\to +\infty}{\lim\sup}\ \HH^\alpha(K_n)\leq \HH^\alpha(K).$$
\end{lemma}

\begin{proof}

For all $n\in\N$, we set $f_n=\mathds{1}_{K_n\setminus K}$. The convergence of $K_n$ to $K$ in the Hausdorff distance implies that $(f_n)$ converges pointwise to $0$. In fact, if $x\in K$, then for all $n, f_n(x)=0$. If $x\notin K$, then there exists $r>0$ such that $B(x,r)\subset X\setminus K$. Fix $n_0\in\N$ such that for all $n\geq n_0, d_H(K_n,K)<r/2$. Then for all $n\geq n_0, x\notin K_n$, thus $f_n(x)=0$. Moreover, since $(K_n)$ converges, it is a bounded sequence, so there exists $\varepsilon >0$ such that for all $n\in\N$, $K_n\subset K^\varepsilon\subset\overline{K^\varepsilon}$, which is a compact set since closed and bounded subsets of an Ahlfors regular space are compact. Then for all $n\in\N, f_n\leq \mathds{1}_{\overline{K^\varepsilon}}$ and the latter function is integrable with respect to $\HH^\alpha$, since compact subsets of an Ahlfors $\alpha$-regular space have finite $\alpha$-dimensional Hausdorff measure.
By the dominated convergence theorem,
$$\HH^\alpha(K_n\setminus K)=\int_X f_n\dd\HH^\alpha \underset{n\to +\infty}{\longrightarrow}0.$$
We conclude by writing 
$\HH^\alpha(K_n)\leq \HH^\alpha(K_n\setminus K)+\HH^\alpha(K)$.
\end{proof}

We will need the following proposition, whose proof can be found in \cite{ambrosio} (Proposition $4.4.14.$)
\begin{proposition}\label{kuratowski} Let $(K_n)_{n\in\N}, K$ be compact sets of a metric space $(X,d)$. If $K_n\overset{H}{\underset{n\to +\infty}{\longrightarrow}}K$ then\begin{enumerate}
\item for all $x\in K$, there exists a sequence $(x_n)_{n\in\N}$ such that $x_n\in K_n$, and $d(x_n,x)\underset{n\to +\infty}{\longrightarrow}0$.
\item for all $x$ such that $x=\underset{k\to +\infty}{\lim}x_{n_k}$ where $(x_{n_k})_{k\in\N}$ is a subsequence of a sequence $(x_n)_{n\in\N}$ such that $x_n\in K_n$, then $x\in K$. 
\end{enumerate}
Moreover the converse is true if $X$ is compact.
\end{proposition}

\section{Product of BPI spaces}\label{product_BPI_section}
We will prove the following

\begin{theorem}\label{product_BPI} Let $(X,d)$ and $(Y,\rho)$ be two BPI spaces of dimension $\alpha$ and $\beta$. If $X$ and $Y$ are both bounded or both unbounded, then the product $X\times Y$ endowed with a product distance is a BPI space of dimension $\alpha+\beta$. 
\end{theorem}

By a \textit{product distance} on the product of two metric spaces $(X,d), (Y,\rho)$ we mean a distance denoted by $\|\cdot\|(d,\rho)$ and defined for all $(x,y),(x',y')\in X\times Y$ by
$$\|\cdot\|(d,\rho)\big((x,y),(x',y')\big)=\|\big(d(x,x'),\rho(y,y')\big)\|,$$ where $\|\cdot\|$ is a norm on $\R^2$. By the equivalence of norms in finite-dimensional vector spaces, all the product distances are biLipschitz equivalent, and since being a BPI space in invariant by biLipschitz maps, it is sufficient to prove Theorem \ref{product_BPI} for one specific product distance.

In this section, we fix two BPI spaces $(X,d)$ and $(Y,\rho)$. Let $\alpha$ denote the dimension of $X$ and $\beta$ the dimension of $Y$. Let $d_\infty$ be the product distance $\|\cdot\|_{\infty}(d,\rho)$, where $\|\cdot\|_{\infty}$ is the $\sup$ norm on $\R^2$.
We remark that if $(x,y)\in X\times Y$, and $r\in (0,\max\{\text{diam}_d X, \text{diam}_\rho Y\}]$, then 
\begin{equation}\label{product_balls}
B_{d_\infty}\big((x,y),r\big)=B_d(x,r)\times B_\rho(y,r).
\end{equation}
First of all, we need to prove that $(X\times Y,d_\infty)$ is Ahlfors regular of dimension $\alpha+\beta$. To do so, the Hausdorff measure $\HH^{\alpha+\beta}_{d_\infty}$ seems to be natural. The problem is that this measure behaves badly with product sets. We define another measure $\m$ on the product $X\times Y$ that is better for measuring product sets. We denote by $\BB_X$ and $\BB_Y$ the Borel $\sigma$-algebra of $X$ and $Y$. For all $A\subset X\times Y$, let
$$\m(A)=\inf\Big\{\sum_{k=1}^{+\infty}\HH^\alpha_d(A_k)\HH^\beta_\rho(B_k),\ A\subset \bigcup_{k=1}^{+\infty}A_k\times B_k,\ A_k\in \BB_X,\ B_k\in \BB_Y\Big\}.$$
In general it is not true that $\HH^{\alpha+\beta}=\m$.

\begin{proposition}\label{product_measure} $\m$ is an outer measure on $X\times Y$ such that for all $A\in \BB_X, B\in \BB_Y$, 
$$\m(A\times B)=\HH^\alpha_d(A)\HH^\beta_\rho(B).$$
Moreover, Borel sets of $(X\times Y,d_\infty)$ are $\m$-measurable.
\end{proposition}

\begin{proof} See \cite{real_analysis} (Theorem $6.2$)
\end{proof}

\begin{proof}[Proof of Theorem \ref{product_BPI}] In this proof, any constant that refers to properties of $X$ or $Y$ is denoted with either an $X$ or a $Y$ in the subscript. First we prove that $(X\times Y,d_\infty)$ is Ahlfors regular of dimension $\alpha+\beta$. 

Suppose that $X$ and $Y$ are both unbounded. For all $(x,y)\in X\times Y$ and all $r>0$, 
$$\m\Big(B_{d_\infty}\big((x,y),r\big)\Big)=\m\big(B_d(x,r)\times B_\rho(y,r)\big)=\HH^\alpha_d\big(B_d(x,r)\big)\HH^\beta_\rho\big(B_\rho(y,r)\big),$$
with \eqref{product_balls} and Proposition \ref{product_measure}, hence
\begin{equation}\label{estimation_balls}
(K_XK_Y)^{-1}r^{\alpha+\beta}\leq\m\Big(B_{d_\infty}\big((x,y),r\big)\Big)\leq K_XK_Yr^{\alpha+\beta}.
\end{equation}
By Lemma \ref{sufficient_condition}, $(X\times Y,d_\infty)$ is Ahlfors regular of dimension $\alpha +\beta$.

If $X$ and $Y$ are both bounded, then the estimate \eqref{estimation_balls} holds for all $(x,y)\in X\times Y$ and all $r\in (0,\min\{\text{diam}_d X, \text{diam}_\rho Y\} ]$. By modifying the Ahlfors regularity constants $K_X, K_Y$, it also holds for all $r\in (0,\max\{\text{diam}_d X, \text{diam}_\rho Y\} ]$, thus $(X\times Y,d_\infty)$ is Ahlfors $(\alpha +\beta)$-regular. 
\\

Fix now two points $(x,y), (x',y')\in X\times Y$ and two radii $r\in (0,\text{diam}_d(X)]$, $t\in (0,\text{diam}_\rho(Y)]$. Let $A_X\subset B_d(x,r),\ A_Y\subset B_\rho(y,r)$ be the two big pieces and $f_X:(A_X,r^{-1}d)\to (B_d(x',t),t^{-1}d)$, $f_Y:(A_Y,r^{-1}\rho)\to (B_\rho (y',t),t^{-1}\rho)$ the two $C_X, C_Y$-biLipschitz maps given by the definition of a BPI space. By Lemma \ref{sufficient_condition}, and Proposition \ref{product_measure} ($A_X$, $A_Y$ are Borel sets), there exists a constant $\theta>0$ such that $\HH^{\alpha+\beta}(A_X\times A_Y)\geq \theta r^{\alpha+\beta}$. Finally, the map 
$$f: (A_X\times A_Y,r^{-1}d_\infty)\to \big(B_{d_\infty}\big((x',y'),t\big),t^{-1}d_\infty\big)$$
defined by $f(x,y)=\big(f_X(x),f_Y(y)\big)$ is a $C$-biLipschitz map, for a constant $C$ depending only on $C_X$ and $C_Y$. 
\end{proof}

We can now prove that $(\R^N,d_\s)$ is a BPI space of dimension $\alpha=\sum s_k^{-1}$. 
\begin{proposition}\label{product_snowflakes}
The metric space $\displaystyle( \R^N,d_\s)$ is a BPI space. 
\end{proposition}

\begin{proof} It it easy to check that if $(X,d)$ is an unbounded BPI space of dimension $\alpha$, then for $0<s<1$, $(X,d^s)$ is an unbounded BPI space of dimension $\alpha /s$. The space $(\R,|\cdot|)$ is an unbounded BPI space of dimension $1$, thus $(\R,|\cdot|^{s_k})$ is an unbounded BPI space of dimension $1/s_k$ for all $k$. By Theorem \ref{product_BPI}, $(\R^N,d_\s)$ is a BPI space of dimension $\alpha=\sum s_k^{-1}$, since $d_\s$ is the $\ell^1$ product distance of the distances $|\cdot|^{s_k}$. 
\end{proof}

\section{Construction of a quotient semi-distance.}\label{new semi-distance}

We construct a semi-distance $d_\RR$ on $\R$ using a shortening technique, by following the article \cite{tapio}. First we define an equivalence relation $\RR$ on $\R$ in a self-similar way. This corresponds to the shortcuts. We then look at the quotient semi-distance $d_\RR$. By construction, $d_\RR\leq |\cdot|^s$, where $s$ is the minimum snowflaking factor. For compactness reason, it is more convenient to work with the subset $[0,1]$ of $\R$, endowed with the semi-distance $d_\RR$. By using a theorem in \cite{tapio}, the quotient space $([0,1]/d_\RR,\overline{d}_\RR)$ is Ahlfors regular of dimension $\alpha=1/s$. We finally prove that the space $([0,1]/d_\RR,\overline{d}_\RR)$ is a BPI space.

\subsection{Motivation: the philosophy of shortcuts.} In $(\R,|\cdot|^s)$, the triangle inequality can be improved if the points are chosen correctly. This is the general idea of the shortcuts' method. Let us explain this. In what follows, every ball is a ball for $|\cdot|^s$. When the centre of a ball of radius $r$ does not matter, we just write $B_r$. 

Let $x,y,p,q\in\R$ be four distinct points. By the triangle inequality, 
\begin{equation}\label{triangle_ineq} |x-y|^s \leq |x-p|^s+|p-q|^s+|q-y|^s.
\end{equation}
Without loss of generality, we may assume that $|x-p|\leq |y-q|$. Suppose that there exists $r>0$ such that $q\in B(p,r)$ and $x,y\notin B(p,r)$. We write $|p-q|^s = C^s r$ where $C\in (0,1)$. 
We have $|x-p|\geq r^{1/s}\geq(1-C)r^{1/s}$ and $|y-q|\geq |y-p|-|p-q|\geq (1-C)r^{1/s}$, so
$$|p-q|\leq \frac{C}{1-C} |x-p|.$$
By using the inequality $(1+t)^s \leq 1+s t$ for $t\geq 0$, we have 
\begin{align}
|x-y|^s 	&\leq \Big(|x-p|+|p-q|+|q-y|\Big)^s \nonumber \\
			&\leq \Big(\frac{1}{1-C}|x-p|+|q-y|\Big)^s \nonumber \\
			&= |q-y|^s\Big(1+\frac{1}{1-C}\frac{|p-x|}{|q-y|}\Big)^s \nonumber\\
			&\leq |q-y|^s +\frac{s}{1-C}|q-y|^{s-1}|p-x| \nonumber \\
			&\leq \frac{s}{1-C}|x-p|^s +|q-y|^s \label{triangle_ineq_better}.
\end{align}

If $C\leq 1-s$, then \eqref{triangle_ineq_better} leads to an improvement of \eqref{triangle_ineq}: the term $|p-q|^s$ disappears. 
The ball $B(p,(1-s)^s r)$ seems to be invisible when going from $x$ to $y$. This motivates the introduction of a shortcut between $p$ and $q$ by identifying them.

\begin{definition}
A metric space $(X,d)$ for which there exists $\lambda\in (0,1)$ such that for all $r>0$, and all balls $B_r$ of radius $r$, there are two points $p,q\in B_r$ satisfying $d(p,q)\geq\lambda r$ and 
\begin{equation}\label{invisible}d(x,y)\leq d(x,p)+d(q,y),\quad \forall x,y\notin B_r
\end{equation}
is called \emph{a space with $\lambda$-invisible pieces}. We say that $(p,q)\in B_r\times B_r$ is \emph{$\lambda$-invisible outside $B_r$}.
\end{definition}

We proved above that $(\R,|\cdot|^s)$ has $\lambda$-invisible pieces for all $0<\lambda <(1-s)^s$. In \cite{tapio}, it was proven that the Heisenberg group as well as any snowflake of an Ahlfors regular space have invisible pieces.

\subsection{Construction of the shortcuts.} Following \cite{tapio}, we construct an equivalence relation $\RR$ that corresponds to the shortcuts.
\\

Let $c$ be a fixed constant. Let $\NN_1$ be a $4\lambda$-separated set, and a $c\lambda$-net for $X$, i.e. 
\begin{equation}\label{initialisation}
\forall x,y\in \NN_1\text{ distinct},\ d(x,y)\geq 4\lambda\text{ and }X = \bigcup_{x\in\NN_1}B_d(x,c\lambda).
\end{equation}
Then for all $x\in\NN_1$, choose two points $p_x,q_x\in B_d(x,\lambda)$ such that $(p_x,q_x)$ is $\lambda$-invisible outside $B_d(x,\lambda)$. Define the shortcuts of level $1$ as $\SS_1=\left\{(p_x,q_x),x\in\NN_1\right\}\cup\left\{(q_x,p_x),x\in\NN_1\right\}\subset X\times X$. By induction, let $\NN_n$ be a $4\lambda^n$-separated set that is also a $c\lambda^n$-net of $X$, such that 
\begin{equation}\label{heredity}
\mathcal{N}_n\subset X\ \setminus \bigcup_{\substack{(p,q)\in\SS_k\\ 1\leq k\leq n-1}}\{p,q\}^{4\lambda^n},\end{equation}
where $\{p,q\}^{4\lambda^n}$ is the $4\lambda^n$-neighborhood of $\{p,q\}$, i.e. $\{p,q\}^{4\lambda^n}=\left\{x\in X, d(\left\{p,q\right\},x)<4\lambda^n\right\}$.

For all $x\in\NN_n$, choose $(p_x,q_x)$ $\lambda$-invisible outside $B_d(x,\lambda^n)$, and then define the level $n$ shortcuts as $\SS_n= \left\{(p_x,q_x),x\in\NN_n\right\}\cup\left\{(q_x,p_x),x\in\NN_n\right\}\subset X\times X$. Finally, define the set of all shortcuts: 
$$\SS=\bigcup_{n\geq 1}\SS_n.$$

\begin{definition}
A set $\SS$ in $X\times X$ constructed as above is called \emph{a set of shortcuts}. An element $(p,q)\in\SS$ is called \emph{a shortcut between $p$ and $q$}. The integer $n$ such that $(p,q)\in\SS_n$ is called \emph{the level of the shortcut $(p,q)$}. 
\end{definition}

With a set of shortcuts $\SS$, we define an equivalence relation $\RR$ on $X$: 
\begin{equation}\label{equiv_relation}x\RR y\Leftrightarrow (x,y)\in \SS\ \text{ or }\ x=y.
\end{equation}
The following proposition is proved in \cite{tapio} (Section $3$).
\begin{proposition}\label{proposition}
Let $\lambda\in (0,1)$ and $(X,d)$ be an Ahlfors $\alpha$-regular space with $\lambda$-invisible pieces. If $\SS$ is a set of shortcuts on $X$ and $\RR$ the equivalence relation defined as in \eqref{equiv_relation}, then the quotient metric space $(X/d_\RR,\overline{d}_\RR)$ is Ahlfors $\alpha$-regular.
\end{proposition}

In our case, we apply Proposition \ref{proposition} to $(X,d)=(\R,|\cdot|^s)$. It is an Ahlfors regular space of dimension $1/s$ with $\lambda$-invisible pieces, for all $0<\lambda < (1-s)^s$. We construct the set of shortcuts $\SS$ in a self-similar way, so that the space $([0,1]/d_\RR,\overline{d}_\RR)$ is a BPI space of dimension $\alpha=1/s$. 

\noindent
Let $l\in\N, h=1/2^l$ and $\mu=h^s$. For all $n\in\N$, we define the level $n$ shortcuts (see Figure \ref{shortcuts})
$$\SS_n=\left\{\Big(h^{n-1}\Big(m+\frac{1}{2}\Big),h^{n-1}\Big(m+\frac{1}{2}\Big)+h^{n+1}\Big), m\in \Z\right\}\subset \R\times \R.$$
Let $\SS=\displaystyle\bigcup_{n\geq 1}\SS_n$. We will see that for $l,c$ large enough, $\SS$ is a set of shortcuts.\\

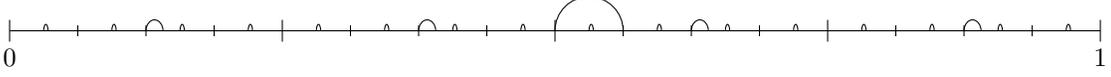
\begin{figure}
\begin{center}
\begin{tikzpicture}[scale=14.5]
\draw (0,0)--(1,0);
\draw (0,-0.01)--(0,0.01);
\draw (1/4,-0.01)--(1/4,0.01);
\draw (1/2,-0.01)--(1/2,0.01);
\draw (3/4,-0.01)--(3/4,0.01);
\draw (1,-0.01)--(1,0.01);
\draw (1/16,-0.005)--(1/16,0.005);
\draw (2/16,-0.005)--(2/16,0.005);
\draw (3/16,-0.005)--(3/16,0.005);
\draw (5/16,-0.005)--(5/16,0.005);
\draw (6/16,-0.005)--(6/16,0.005);
\draw (7/16,-0.005)--(7/16,0.005);
\draw (9/16,-0.005)--(9/16,0.005);
\draw (10/16,-0.005)--(10/16,0.005);
\draw (11/16,-0.005)--(11/16,0.005);
\draw (13/16,-0.005)--(13/16,0.005);
\draw (14/16,-0.005)--(14/16,0.005);
\draw (15/16,-0.005)--(15/16,0.005);
\draw (1/2,0) arc (180:0:1/32);
\draw (1/8,0) arc (180:0:1/128 and 1/100);
\draw (3/8,0) arc (180:0:1/128 and 1/100);
\draw (5/8,0) arc (180:0:1/128 and 1/100);
\draw (7/8,0) arc (180:0:1/128 and 1/100);
\draw (1/32,0) arc (180:0:1/512 and 1/170);
\draw (3/32,0) arc (180:0:1/512 and 1/170);
\draw (5/32,0) arc (180:0:1/512 and 1/170);
\draw (7/32,0) arc (180:0:1/512 and 1/170);
\draw (9/32,0) arc (180:0:1/512 and 1/170);
\draw (11/32,0) arc (180:0:1/512 and 1/170);
\draw (13/32,0) arc (180:0:1/512 and 1/170);
\draw (15/32,0) arc (180:0:1/512 and 1/170);
\draw (17/32,0) arc (180:0:1/512 and 1/170);
\draw (19/32,0) arc (180:0:1/512 and 1/170);
\draw (21/32,0) arc (180:0:1/512 and 1/170);
\draw (23/32,0) arc (180:0:1/512 and 1/170);
\draw (25/32,0) arc (180:0:1/512 and 1/170);
\draw (27/32,0) arc (180:0:1/512 and 1/170);
\draw (29/32,0) arc (180:0:1/512 and 1/170);
\draw (31/32,0) arc (180:0:1/512 and 1/170);
\draw (0,-0.025) node{$0$};
\draw (1,-0.025) node{$1$};
\end{tikzpicture}
\caption{The shortcuts $\SS_1, \SS_2$ and $\SS_3$ in $[0,1]$ for $l=2, h=1/4$.}
\label{shortcuts}
\end{center}
\end{figure}

Following the notations defined above, we set $\NN_n=\left\{h^{n-1}\Big(m+\frac{1}{2}\Big),m\in\Z\right\}$. If $x,y\in\NN_n$ are distinct, then $|x-y|^s\geq\mu^{n-1}.$ Moreover, for all $x\in\R$, there exists $m\in\Z$ such that $\left|x-h^{n-1}(m+1/2)\right|^s=\mu^{n-1}|x/h^{n-1}-1/2-m|^s\leq 2^{-s}\mu^{n-1}$. Since $\NN_n$ must be a $4\lambda^n$-separated set and a $c\lambda^n$-net for $X$, we obtain the inequalities
\begin{equation}\label{cond_1}4\lambda^n\leq\mu^{n-1}\leq 2^sc\lambda^n.
\end{equation}
If $x\in\NN_n$, the closest point $p$ to $x$ for which there exists $q$ such that $(p,q)\in\bigcup_{1\leq j \leq n-1}S_j$ verifies $|x-p|^s=\mu^{n-1}(1/2-h)^s=\mu^{n-2}(2^{l-1}-1)^s$. The condition \eqref{heredity} implies that 
\begin{equation}\label{cond_2}
4\lambda^n\leq\mu^{n-2}(2^{l-1}-1)^s.
\end{equation}
Finally, every couple $(p_m,q_m)\in\SS_n$ has to be $\lambda$-invisible outside $B_{|\cdot|^s}(h^{n-1}(m+1/2),\lambda^n)$, so
\begin{equation}\label{cond_3}
\lambda^{n+1}\leq \mu^{n+1}\leq (1-s)^s\lambda^n.
\end{equation}
We see that if $\lambda=\mu$ and $l,c$ are sufficiently large such that 
$$\frac{1}{2^sc}\leq\mu\leq \min\Big\{\frac{1}{4},(1-s)^s,\frac{1}{2}(2^{l-1}-1)^{s/2}\Big\},$$
then the conditions \eqref{cond_1}, \eqref{cond_2} and \eqref{cond_3} are verified. Let us henceforth fix the constants $l$ and $c$. With this set of shortcuts $\SS$, we define as above the equivalence relation $\RR$ on $\R$ by 
$$x\RR y\Leftrightarrow (x,y)\in\SS \text{ or }x=y.$$ By Proposition \ref{proposition}, the space $(\R/d_\RR,\overline{d}_\RR)$ is Ahlfors regular of dimension $\alpha=1/s$. 
\\

In the sequel, it will be interesting to work on $[0,1]$ endowed with this semi-distance $d_\RR$. A priori, if $\RR_{[0,1]}$ denotes the equivalence relation $\RR$ restricted on the subset $[0,1]\times [0,1]$ (that is one takes only shortcuts in $[0,1]$), then $d_\RR\leq d_{\RR_{[0,1]}}$, because there are more itineraries for the relation $\RR$ than for the relation $\RR_{[0,1]}$. Actually we will see in Corollary \ref{egalite} that for all $x,y\in [0,1], d_\RR(x,y)=d_{\RR_{[0,1]}}(x,y)$.

By Proposition \ref{proposition} and by what we observed above, the space $([0,1]/d_{\RR_{[0,1]}},\overline{d}_{\RR_{[0,1]}})$ is also Ahlfors regular of dimension $\alpha$.

Until the end of the article, the symbol $\RR$ always denotes the equivalence relation on $\R$ we just constructed.

\subsection{The metric space $([0,1],d_\RR)$ is a BPI space.}\label{d_R BPI}

In order to have compactness, we will prove that the subspace $[0,1]$, endowed with $d_\RR$ is a BPI space. To do so, we prove that $d_\RR=d_{\RR_{[0,1]}}$ on $[0,1]$ (Corollary \ref{egalite}), thus $([0,1]/d_\RR,\overline{d}_\RR)$ is Ahlfors regular of dimension $\alpha =1/s$. Then we prove that any two balls possess big pieces that are biLipschitz equivalent for the rescaled distances.

The notion of BPI space has been defined for a metric space, but $d_\RR$ is only a semi-distance on $[0,1]$. Let $B(\pi(x),r)\subset ([0,1]/d_\RR,\overline{d}_\RR)$ be a ball. By definition, we have $B(\pi(x),r)=\pi\big(B(x,r)\big)$ where $\pi : ([0,1],d_\RR)\to ([0,1]/d_\RR,\overline{d}_\RR)$ is the canonical projection, and $B(x,r)\subset ([0,1],d_\RR)$. Moreover, $\pi$ is an isometry by definition. To prove that $([0,1]/d_\RR,\overline{d}_\RR)$ is a BPI space, it is thus sufficient to prove that there exist constants $C,\theta$ so that for each pair of balls $B(x,r), B(y,t)$ in $([0,1],d_\RR)$, there is a closed subset $A\subset B(x,r)$ with $\HH^\alpha_{d_\RR}(A)\geq \theta r^\alpha$ and a $C$-biLipschitz map $f:(A,r^{-1}d_\RR)\to \big(B(y,t),t^{-1}d_\RR\big)$. 
\\

Until the end of the section, each ball $B(x,r)$ is a ball for the semi-distance $d_\RR$. 
For this paragraph, we introduce the following definition for notational convenience. 

\begin{definition} An interval $I=[a,b]$ is called \emph{an interval without shortcut at the ends} if for all $x\in\R, (a,x)\notin \SS$ and $(x,b)\notin \SS$. 
\end{definition}

\begin{lemma}\label{existence} For all $x\in\R$, and $r>0$, there exists $I\subset B(x,r)$ an interval of the form $I=[h^nm,h^n(m+1)]$, without shortcut at the ends, where $n=1+\lceil\frac{\ln r}{\ln \mu}\rceil$ and $ m\in\Z$.
\end{lemma}

\begin{proof}
Since $d_\RR\leq |\cdot|^s$, we have $]x-r^{1/s},x+r^{1/s}[\subset B(x,r)$. By definition of $n$, $h^{n-1}\leq r^{1/s}$, so $B(x,r)$ contains an interval of (Euclidean) length $h^{n-1}$. In this interval, we can find $h^{-1}-1$ intervals of the form $[h^nm,h^n(m+1)]$, where $m\in\Z$. Among these intervals, at least one suits.
\end{proof}


Let us make the following easy remark, that will be useful later: if $(p,q)\in\SS$ with $p<q$, is a shortcut of level less than or equal to $n$, then there exists $m\in\Z$ such that $p=h^nm$ (and then $q=h^{n+1}(m2^{l}+1)$). Moreover, the converse is true: if $(p,q)\in\SS, p<q$, and $p=h^nm$, then the level of the shortcut $(p,q)$ is less than or equal to $n$. 
We can also say something about $q$: if $(p,q)\in\SS, p<q$ and $q=h^{n+1}m$, then the level of the shortcut $(p,q)$ is less than or equal to $n$. 
\begin{lemma}\label{in_out} Let $n\geq 0, m\in\Z$ and $I=[h^nm,h^n(m+1)]$ be an interval without shortcut at the ends. Then $I$ contains no shortcut of level less than or equal to $n$. Moreover, for all $(p,q)\in\SS,\ p\in I\Leftrightarrow q\in I$. 
\end{lemma}

\begin{proof}
The fact that $I$ contains no shortcut of level less than or equal to $n$ is easy with the remark made above, and with the assumption that $I$ is an interval without shortcut at the ends. 

Then, let $(p,q)\in\SS, p<q$ be a shortcut of level $n' > n$. Write $p=h^{n'}m'$, $q=h^{n'}m'+h^{n'+1}$ with $m'\in\Z$. Suppose that $p\in I$, that is
$h^nm<h^{n'}m' <h^{n}(m+1)$, i.e. $m<h^{n'-n}m' <m+1.$ If one cuts the interval $[m,m+1]$ into equal intervals of length $h^{n'-n}$, then one sees that $|m+1-h^{n'-n}m'|\geq h^{n'-n}>h^{n'-n+1}$, so $m<h^{n'-n}m'+h^{n'-n+1} <m+1$, which means that $q\in I$. A similar argument proves that $q\in I\Rightarrow p\in I$. 
\end{proof}

\begin{lemma}\label{adjacent} Let $n\geq 0, m\in\Z$ and $I=[h^nm,h^n(m+1)]$ be an interval without shortcut at the ends. Then either $[h^n(m-1),h^nm]$ or $[h^n(m+1),h^n(m+2)]$ is an interval without shortcut at the ends. 
\end{lemma}
\begin{proof}
Suppose not. Then, there exist $x,y\in\R$ such that $(h^n(m-1),x)\in\SS$ and $\big(y,h^n(m+2)\big)\in\SS$. Many variations are possible: 
\begin{enumerate}
\item If $h^n(m-1)< x$ and $h^n(m+2)<y$, then these two shortcuts have level less than or equal to $n$.
\item If $h^n(m-1)<x$ and $y<h^n(m+2)$, then the level of $(h^n(m-1),x)$ is less than or equal to $n$ and the level of $\big(y,h^n(m+2)\big)$ is less than or equal to $n-1$.
\item If $x<h^n(m-1)$ and $h^n(m+2)<y$, then the level of $\big(x,h^n(m-1)\big)$ is less than or equal to $n$ and the level of $(h^n(m+2),y)$ is less than or equal to $n$.
\item If $x<h^n(m-1)$ and $y<h^n(m+2)$, then these two shortcuts have level less than or equal to $n-1$.
\end{enumerate}
In each case, this is impossible, because the distance between the two shortcuts is too small. 
\end{proof}

\begin{lemma}\label{restriction_distance} Let $n\geq 0, m\in\Z$ and $I=[h^nm,h^n(m+1)]$ be an interval without shortcut at the ends. Denote by $\RR_I$ the equivalence relation on $I$ which is the restriction of $\RR$ to the subset $I\times I$. Then for all $x,y\in I$,
$$d_\RR(x,y)=d_{\RR_I}(x,y).$$
\end{lemma}

\begin{proof}
The inequality $d_\RR\leq d_{\RR_I}$ is easy. Let us prove the other inequality. Let $x,y\in I$. Let $(x_0,y_0,\dots,x_n,y_n)$ be an itinerary from $x$ to $y$. Suppose that this itinerary gets out of $I$. We will construct another itinerary that stays in $I$ and that is shorter. Let $a=\min\left\{k\in \left\{0,\dots,n\right\}, x_k\notin I\right\}$ and $b=\max\left\{k\in \left\{0,\dots,n\right\}, x_k\notin I\right\}$. Since $x\in I$ and $x\RR x_0$, by Lemma \ref{in_out}, $a\geq 1$. 
\begin{enumerate}
\item If $x_a$ and $x_b$ are in the same connected component of $\R\setminus I$, say $x_a,x_b < h^nm$, we first remark that $y_{a-1}\notin I$ and $y_b\in I$. In fact, $y_{a-1}\RR x_a$, but by Lemma \ref{in_out}, since $x_a\notin I$, we have $y_{a-1}\notin I$. The same argument works for $y_b$. Then the itinerary 
$$(x_0,y_0,\dots,x_{a-1},h^nm,h^nm,y_{b},\dots, x_n,y_n)$$ 
stays in $I$, and is shorter than $(x_0,y_0,\dots,x_n,y_n)$. The same construction works if $h^n(m+1)<x_a,x_b$. 

\item If $x_a$ and $x_b$ are not in the same connected component of $\R\setminus I$, say $x_a<h^nm<h^n(m+1)<x_b$. We may suppose that there exists an integer $j\in\left\{a,\dots,b-1\right\}$ such that $y_j$ and $x_{j+1}$ are not in the same connected component of $\R\setminus I$ (that means that the itinerary follows a shortcut that steps over $I$, see Figure \ref{across}). In fact if this is not so, we can conclude as in the first case. We remark that if $n=0$, then this case is excluded.

 Let $J$ be an interval without shortcut at the ends, adjacent to $I$, given by Lemma \ref{adjacent}. Without loss of generality, we may suppose that $J=I+h^n$. The following remark is easy : for all $(p,q)\in\SS,\ (p,q)\in I\times I\Leftrightarrow (p+h^n,q+h^n)\in J\times J$.

\begin{figure}
\begin{center}
\begin{tikzpicture}
\draw (1.5,0)--(2.5,0);
\draw (3,0)--(9,0);
\draw [dotted] (9,0)--(10.5,0);
\draw [dotted] (2.5,0)--(3,0);
\draw (2,-0.1)--(2,0.1);
\draw (4,-0.1)--(4,0.1);
\draw (6,-0.1)--(6,0.1);
\draw (8,-0.1)--(8,0.1);
\draw (10,-0.1)--(10,0.1);
\draw [decorate,decoration={brace,amplitude=7pt},yshift=4.55pt](4,0) -- (6,0) node[black,midway,yshift=0.5cm] {$I$};
\draw [decorate,decoration={brace,amplitude=7pt},yshift=4.55pt](6,0) -- (8,0) node[black,midway,yshift=0.5cm] {$J$};
\draw [dotted] (2,0) arc (180:140:8 and 2);
\draw [dotted] (10,0) arc (0:40:8 and 2);
\draw (4,0) node[below]{\footnotesize$ h^nm$};
\draw (6,0) node[below]{\footnotesize$h^n(m+1)$};
\draw (2,0) node[below]{\footnotesize$y_j$};
\draw (10,0) node[below]{\footnotesize$x_{j+1}$};
\end{tikzpicture}
\end{center}
\caption{The shortcut $(y_j,x_{j+1})$ across $I$.}
\label{across}
\end{figure}
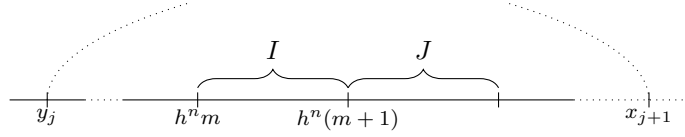

Since $J$ has no shortcut at the ends, there exists $i$ such that $x_i\notin J\cup I$, $y_i\in J$. 
 Define the new itinerary in three parts.
 \begin{enumerate}[label=(\roman*)]
 \item The first part is $(x_0,y_0,\dots,x_{a-1},h^nm)$.
 \item The second part is $(h^nm,x_b-h^n,y_{b-1}-h^n,x_{b-1}-h^n,\dots,y_i-h^n,h^n(m+1))$.
 \item The last part is $(h^n(m+1),y_b,\dots,x_n,y_n)$. 
 \end{enumerate}
 
 In Figures \ref{original} and \ref{modified}, we represent the original and modified itineraries, the thickest parts are the parts of the itinerary where one has to walk, and the arcs are the shortcuts. 

\begin{figure}
\begin{center}
\begin{tikzpicture}
\draw (0.5,0)--(8.7,0);
\draw [dotted] (9.5,0)--(8.7,0);
\draw [dotted] (0,0)--(0.5,0);
\draw (2,-0.1)--(2,0.1);
\draw (5,-0.1)--(5,0.1);
\draw (8,-0.1)--(8,0.1);
\draw [thick] (1.6,0)--(2.3,0); 
\draw [thick] (0.5,0)--(0.7,0); 
\draw [thick] (5.6,0)--(4.5,0); 
\draw [thick] (7.4,0)--(8.3,0); 
\draw (3.3,-0.05)--(3.3,0.05); 
\draw (3.7,-0.05)--(3.7,0.05); 
\draw (2.3,-0.05)--(2.3,0.05); 
\draw (1.6,-0.05)--(1.6,0.05); 
\draw (0.7,-0.05)--(0.7,0.05); 
\draw (4.5,-0.05)--(4.5,0.05); 
\draw (5.6,-0.05)--(5.6,0.05); 
\draw (6.0,-0.05)--(6,0.05); 
\draw (8.3,-0.05)--(8.3,0.05); 
\draw (7.4,-0.05)--(7.4,0.05); 
\draw [decorate,decoration={brace,amplitude=7pt,mirror},yshift=-4.55pt](2,0) -- (5,0) node[black,midway,yshift=-0.5cm] {$I$};
\draw [decorate,decoration={brace,amplitude=7pt,mirror},yshift=-4.55pt](5,0) -- (8,0) node[black,midway,yshift=-0.5cm] {$J$};
\draw (3.3,0.22) node{\tiny$x$};
\draw (3.7,0.2) node{\tiny$y$};
\draw (2.9,0.22) node{\tiny$\dots$};
\draw (4.1,0.22) node{\tiny$\dots$};
\draw (2.35,0.22) node{\tiny$x_{a-1}$};
\draw (1.6,0.22) node{\tiny$y_{a-1}$};
\draw (0.7,0.22) node{\tiny$x_{a}$};
\draw (4.5,0.22) node{\tiny$y_{b}$};
\draw (5.6,0.22) node{\tiny$x_{b}$};
\draw (7.4,0.22) node{\tiny$y_{i}$};
\draw (8.3,0.22) node{\tiny$x_{i}$};
\draw (6.1,0.22) node{\tiny$y_{b-1}$};
\draw (6.8,0.22) node{\tiny$\dots$};
\draw [dotted] (0.2,0) arc (180:140:8 and 2);
\draw [dotted] (9,0) arc (0:40:8 and 2);
\draw (0.7,0) arc (180:360:0.45 and 0.2);
\draw (5.6,0) arc (180:360:0.2 and 0.15);
\draw [<-](2,0.6)--(3.3,0.6);
\draw [<-](5.2,0.6)--(6.5,0.6);
\draw [->](0.6,0.4) arc (165:150:8 and 2);
\draw [->](7.6,0.95) arc (33:17:8 and 2);
\end{tikzpicture}
\caption{The original itinerary.}
\label{original}
\end{center}
\end{figure}
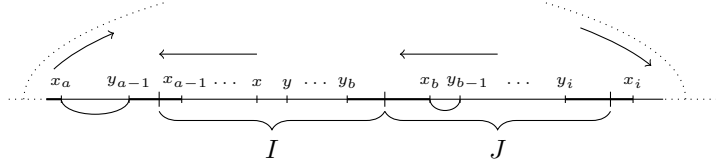

\begin{figure}
 \begin{center}
\begin{tikzpicture}
\draw (2,0)--(5,0);
\draw (2,-0.1)--(2,0.1);
\draw (5,-0.1)--(5,0.1);
\draw [thick] (2,0)--(2.3,0); 
\draw [thick] (5,0)--(4.5,0); 
\draw (3.3,-0.05)--(3.3,0.05); 
\draw (3.7,-0.05)--(3.7,0.05); 
\draw (2.3,-0.05)--(2.3,0.05); 
\draw (4.5,-0.05)--(4.5,0.05); 
\draw [decorate,decoration={brace,amplitude=7pt,mirror},yshift=-4.55pt](2,0) -- (5,0) node[black,midway,yshift=-0.5cm] {$I$};
\draw (3.3,0.22) node{\tiny$x$};
\draw (3.7,0.2) node{\tiny$y$};
\draw (2.9,0.22) node{\tiny$\dots$};
\draw (4.1,0.22) node{\tiny$\dots$};
\draw (2.35,0.22) node{\tiny$x_{a-1}$};
\draw (4.5,0.22) node{\tiny$y_{b}$};
\draw [->] (3.3,0.6)--(2,0.6);
\draw [->] (5,0.6)--(3.7,0.6);
\draw [dotted] [->] (2,-0.6)--(2,-1.3);
\draw [dotted] [->] (5,-1.3)--(5,-0.6);
\draw (2,-2)--(5,-2);
\draw (2,-2.1)--(2,-1.9);
\draw (5,-2.1)--(5,-1.9);
\draw [decorate,decoration={brace,amplitude=7pt,mirror},yshift=-4.55pt](2,-2) -- (5,-2) node[black,midway,yshift=-0.5cm] {$J-h^n=I$};
\draw (2.6,-2.05)--(2.6,-1.95); 
\draw (3.0,-2.05)--(3,-1.95); 
\draw (4.4,-2.05)--(4.4,-1.95); 
\draw (3.8,-2+0.22) node{\tiny$\dots$};
\draw (2.6,-2) arc (180:360:0.2 and 0.15);
\draw [thick] (2,-2)--(2.6,-2);
\draw [thick] (4.4,-2)--(5,-2);
\draw [->] (2.5,-1.6)--(4.5,-1.6);
\end{tikzpicture}
\caption{The modified itinerary, which stays in $I$. }
\label{modified}
\end{center}
\end{figure}
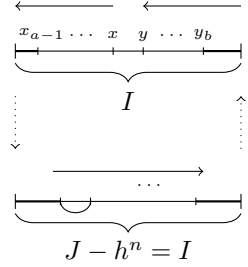

This modified itinerary stays in $I$, and is shorter than the original one. This construction works similarly if $x_b<h^nm<h^n(m+1)<x_a$, or if $J=I-h^n$. 
\end{enumerate}
We have proved that for any itinerary from $x$ to $y$, there exists a shorter itinerary between $x$ and $y$ that stays in $I$. Therefore $d_{\RR_I}\leq d_{\RR}$. 
\end{proof} 
  
\begin{coro}\label{egalite} For all $x,y\in [0,1], d_\RR(x,y)=d_{\RR_{[0,1]}}(x,y)$.
\end{coro}
 
We can now prove that $([0,1],d_\RR)$ is a BPI space of dimension $\alpha=1/s$. By Proposition \ref{proposition}, we know that it is Ahlfors regular of dimension $\alpha$. Let $x,y\in [0,1]$, $r,t\in (0,\text{diam}_{d_\RR}([0,1])]$, $n=1+\lceil \ln r/\ln\mu \rceil$ and $n'=1+\lceil \ln t/\ln\mu \rceil$. Let $I=[h^nm,h^n(m+1)]\subset B(x,r)$ and $I'=[h^{n'}m',h^{n'}(m'+1)]\subset B(y,t)$ be given by Lemma \ref{existence}. We remark that $n,n' \geq 1$, since $r,t\leq 1$.

 Define $f:I\to I'$ by $f(x)=h^{n'-n}x+h^{n'}(m'-m)$. Then $f$ is bijective and compatible with the shortcuts, that is $x\RR_I y\Leftrightarrow f(x)\RR_{I'} f(y)$, and $|f(x)-f(y)|^s=\mu^{n'-n}|x-y|^s$, so by Lemma \ref{lipschitz_map}, for all $x,y\in I$, 
$$
d_{\RR_{I'}}\big(f(x),f(y)\big)=\mu^{n'-n}d_{\RR_I}(x,y),
$$
thus by Lemma \ref{restriction_distance},
$$d_{\RR}\big(f(x),f(y)\big)=\mu^{n'-n}d_{\RR}(x,y).$$
The inequalities
$\displaystyle\frac{\ln(t/r)}{\ln\mu}-1\leq n'-n\leq \frac{\ln(t/r)}{\ln\mu}+1$
imply that
$\displaystyle\frac{t}{r}\mu\leq \mu^{n'-n}\leq \frac{t}{r}\mu^{-1}.$
Therefore $f:(I,r^{-1}d_\RR)\to (I',t^{-1}d_\RR)$ is a $\mu^{-1}$-biLipschitz map. 

We finally have to estimate $\HH^\alpha_{d_\RR}(I)$. The map $g:I\to [0,1]$ defined by $g(x)=h^{-n}x-m$ is a bijection that verifies 
$$d_\RR\big(g(x),g(y)\big)=\mu^{-n}d_\RR(x,y),$$
by Lemma \ref{lipschitz_map} and Lemma \ref{restriction_distance} since $g$ is bijective and compatible with the shortcuts. Then $\HH^\alpha_{d_\RR}(I)=\mu^{n\alpha}\HH^\alpha_{d_\RR}([0,1]).$ By definition of $n$, $\mu^{n\alpha}\geq \mu^{2\alpha}r^\alpha$. Finally, if we set $\theta = \mu^{2\alpha}\HH^\alpha_{d_\RR}([0,1])$, then $\HH^\alpha_{d_\RR}(I)\geq \theta r^\alpha$. 

We have proved that $([0,1],d_\RR)$ is a BPI space of dimension $\alpha=1/s$.

\section{Lipschitz functions between $([0,1],d_\RR)$ and $(\R,|\cdot|^s)$. }\label{lipschitz_function}

The following proposition deals with Lipschitz functions from $([0,1],d_\RR)$ to $(\R,|\cdot|^s)$.
\begin{proposition}\label{no_lipschitz_function} Any Lipschitz function $f:([0,1],d_\mathcal{R})\to(\R,|\cdot|^s)$ is constant.
\end{proposition}

\begin{proof}
Suppose not. Let $f:([0,1],d_\mathcal{R})\to(\R,|\cdot|^s)$ be a non constant Lipschitz map. Since $d_\RR\leq |\cdot|^s$, $f:([0,1],|\cdot|^s)\to (\R,|\cdot|^s)$ is Lipschitz, and hence $f:([0,1],|\cdot|)\to (\R,|\cdot|)$ is also a Lipschitz map. By the Rademacher Theorem, $f$ is differentiable almost everywhere. Since $f$ is a non constant Lipschitz map, $f'$ does not vanish almost everywhere. Hence there exists a subset $A$ of $[0,1]$ with positive measure such that for all $x\in A, f'(x)\neq 0$. For $x_0\in A$ and $p\in [0,1]$, 
\begin{equation}\label{derivative}
\left|\frac{1}{\varepsilon}\Big(f(x_0+\varepsilon p)-f(x_0)\Big)-f'(x_0)p\right|^s\underset{\varepsilon\to 0}{\longrightarrow}0.
\end{equation}
Let $p,q\in [0,1]$ and $\varepsilon>0$ sufficiently small so that $x_0+\varepsilon p, x_0+\varepsilon q\in [0,1]$. Then
\begin{align}
|f'(x_0)p-f'(x_0)q|^s & \leq \left|f'(x_0)p-\frac{1}{\varepsilon}\Big(f(x_0+\varepsilon p)-f(x_0)\Big)\right|^s \label{ineq} \\ 
					& \quad\quad+ \varepsilon^{-s}|f(x_0+\varepsilon p)-f(x_0+\varepsilon q)|^s \nonumber \\
					& \qquad\qquad + \left|\frac{1}{\varepsilon}\Big(f(x_0+\varepsilon q)-f(x_0)\Big)-f'(x_0)q\right|^s \nonumber.
\end{align}
By \eqref{derivative}, the first and third terms of \eqref{ineq} tend to $0$ when $\varepsilon\to 0$. Since $f$ is Lipschitz, there exists $L>0$ such that $\varepsilon^{-s}|f(x_0+\varepsilon p)-f(x_0+\varepsilon q)|^s \leq L\varepsilon^{-s}d_\RR(x_0+\varepsilon p,x_0+\varepsilon q)$.

For all $n\in\N$, $m\in\Z$, the map $x\to h^n(x+m)$ defined on $\R$ is a $h^{sn}$-Lipschitz map for $d_\RR$. In fact, it follows on from Lemma \ref{lipschitz_map}, since $x\RR y\Rightarrow h^n(x+m)\RR h^n(y+m)$. Then, 
\begin{align}
h^{-sn}d_\RR(x_0+h^np,x_0+h^nq) 	
							&\leq d_\RR(x_0/h^n+p-m,x_0/h^n+q-m) \nonumber \\							
							&\leq d_\RR(x_0/h^n-m+p,p)+d_\RR(p,q)+d_\RR(q,q+x_0/h^n-m) \label{ineq2}.
\end{align}
Since $d_\RR\leq |\cdot|^s$, the first and third terms of \eqref{ineq2} are less than or equal to 
$$\left|\frac{x_0}{h^n}-m\right|^s=\left|2^{ln}x_0-m\right|^s.$$

The following is a well-known fact of base-$2$ expansion of real numbers: There exists a Borel set $U \subset [0,1]$ with $\LL^1(U) = 1$ such that for every point $x\in U$, there exists a sequence $(n_p)_{p\in\N}$ of integers such that $n_p\longrightarrow+\infty$ when $p\to +\infty$, and for all $p\in\N$, $x_{ln_p+1}=\dots =x_{2ln_p}=0$, where 
$x=\sum_{k\geq 0}x_k2^{-k}$
is the standard binary representation of $x$. In fact, if $x \in U$, then for all $p\in\N$,
\begin{align*}
2^{ln_p}x&=m_p+\sum_{k=2ln_p+1}^{+\infty}x_k2^{ln_p-k},\quad\text{ where }m_p\in\Z \\
		&\leq m_p+2^{ln_p}\sum_{k=2ln_p+1}^{+\infty}2^{-k} \\
		&= m_p+\frac{1}{2^{ln_p}}, 
\end{align*}
so $2^{ln_p}-m_p\underset{p\to +\infty}{\longrightarrow} 0$.

Let us finish the proof of the proposition. Since $A$ has positive measure, there exists a point $x_0\in A \cap U$. With the inequalities \eqref{ineq}, \eqref{ineq2} we can conclude that 
$$|f'(x_0)p-f'(x_0)q|^s\leq L d_\RR(p,q),$$
but this is impossible since $f'(x_0)\neq 0$ (take $(p,q)\in\SS$ so that $d_\RR(p,q)=0$ but $|p-q|>0$). Thus $f$ is constant. 
\end{proof}

\section{Blow-up}\label{blow_up}
In the sequel, topological properties (closed sets, compact sets, etc) are related to the Euclidean topology on $\R^N$, which is the same as the topology induced by $d_\s$. The balls will be balls for the distance $d_\s$ and the measure $\HH^\alpha$ will always refer to $\HH^\alpha_{d_\s}$, where $\alpha=\sum s_k^{-1}$ is the dimension of the BPI space $(\R^N, d_\s)$. 
Recall that $s$ is the minimum snowflaking factor, and $L$ the minimally snowflaked layer. We define a semi-distance $d_{\s,\RR}$ on $\R^N$ by modifying $d_\s$ on the minimally snowflaked layer $L$: 
$$\forall x,y\in \R^N,\ d_{\textsf{s},\RR}(x,y)=\sum_{k\in L}d_\RR(x_k,y_k)+\sum_{k\notin L}|x_k-y_k|^{s_k},$$
where $d_\RR$ is the semi-distance defined on $\R$ in Section 
\ref{new semi-distance}, by the shortcuts method. The topology induced by $d_{\s,\RR}$, for which a basis is given by the open balls $\{y\in \R^N, d_{\s,\RR}(x,y)<r\}$, is the Euclidean topology.

\begin{proposition} The quotient space $([0,1]^N/d_{\s,\RR}, \overline{d}_{\s,\RR})$ is a BPI space of dimension $\alpha$.
\end{proposition}

\begin{proof}The quotient space $([0,1]^N/d_{\s,\RR}, \overline{d}_{\s,\RR})$ is the product of $([0,1],|\cdot|^{s_k})$ for all $k\notin L$ and $([0,1]/d_\RR,\overline{d}_\RR)$ for all $k\in L$, which are bounded BPI spaces. Thus we may apply Theorem \ref{product_BPI} to conclude.\end{proof}

We want to prove the following theorem, by a blow-up technique.
\begin{theorem}\label{blow-up} Let $A\subset [0,1]^N$ be a closed subset, and $g:(A,d_{\s,\RR})\to (\R^N,d_\s)$ a Lipschitz map such that $\HH^{\alpha}\big(g(A)\big)>0$. Then there exists a Lipschitz map $f:([0,1]^N,d_{\s,\RR})\to (\R^N,d_\s)$ such that $\HH^\alpha\big(f([0,1]^N)\big)>0$. 
\end{theorem}
From now on, we fix a closed (and thus compact) subset $A$ of $[0,1]^N$, and an $L$-Lipschitz map $g:(A,d_{\s,\RR})\to (\R^N,d_\s)$.

For $\delta>0$, we define $\text{dil}_{\delta}(z)=(\delta^{1/s_1}z_1,\dots,\delta^{1/s_N}z_N)$ where $z=(z_1,\dots,z_N)\in\R^N$. For all $j=(j_1,\dots, j_N)\in \Z^N$. The map $\text{dil}_\delta$ is a similitude for $d_\s$: for all $x,y\in \R^N,$ 
$$d_\s\big(\text{dil}_\delta(x),\text{dil}_\delta(y)\big)=\delta d_\s(x,y).$$

For all $i\in\N$, we set 
$$I^i_{j}=\text{dil}_{h^{is}}([0,1]^N+j)=\prod_{k=1}^N\big[h^{i\frac{s}{s_k}}j_k, h^{i\frac{s}{s_k}}(j_k+1)\big].$$
The $I^i_j$ will be called "cubes" in the sequel even though it would be more correct to call them parallelepipeds. The family $\{I^i_j,\ i\in\N, j\in\Z^N\}$ is not a family of nested cubes, which means that two cubes $I^i_j, I^{i'}_{j'}$ with $i\neq i'$ might overlap. 

If we set $f_{i,j}(x)=\text{dil}_{h^{is}}(x+j),$ then $I^i_j=f_{i,j}([0,1]^N)$. We also define 
$$Adm=\left\{(i,j)\in\N\times\Z^N,\ I^i_j\subset [0,1]^N\text{ and } f_{i,j}:([0,1]^N,d_{\s,\RR})\to (I^i_j, h^{-is}d_{\s,\RR})\text{ is an isometry}\right\},$$
the set of indices for which the corresponding cube is incuded in the unit cube and is similar to it for the distance $d_{\s,\RR}$. 

Recall that by construction, for all $(i,j)\in\N\times\Z^N$ such that $I^i_j\subset [0,1]^N$, there exists $j'\in\Z^N$ such that $(i+1,j')\in Adm$ and $I^{i+1}_{j'}\subset I^i_j$.

For all $(i,j)\in Adm$, such that $A\cap I^i_j\neq\varnothing$, choose $z_{i,j}\in A\cap I^i_j$ and define

$$\begin{array}{crcl}
g_{i,j}: & \big(f_{i,j}^{-1}(A\cap I^i_j),d_{\s,\RR}\big) &\longrightarrow & (\R^N, d_\s) \\ 
	  & x & \longmapsto & \text{dil}_{h^{-is}}\big(g(f_{i,j}(x))-g(z_{i,j})\big)
\end{array}.$$

\begin{lemma} The maps $\left\{g_{i,j}, {(i,j)\in Adm}\right\}$ are Lipschitz, with uniformly bounded Lipschitz constants. 
\end{lemma}

\begin{proof} Let $(i,j)\in Adm$ such that $A\cap I^i_j\neq\varnothing$. For all $x,y\in f_{i,j}^{-1}(A\cap I^i_j)$,
\begin{align}
d_\s\big(g_{i,j}(x),g_{i,j}(y)\big) &= h^{-is}d_\s\big(g(f_{i,j}(x)),g(f_{i,j}(y))\big) \nonumber \\ 
				&\leq Lh^{-is}d_{\s,\RR}\big(f_{i,j}(x),f_{i,j}(y)\big) \nonumber \\
				&\leq L d_{\s,\RR}(x,y)\nonumber. 
\end{align}
\end{proof}

For all $m\in\N$ and all indices $i\in\N$, we set 
$$\EE^i_m=\displaystyle\left\{j\in\Z^N, \frac{\HH^\alpha(A\cap I^i_j)}{\HH^\alpha(I^i_j)}<1-\frac{1}{m}\right\}.$$ In order to prove Theorem \ref{blow-up}, we need to find a sequence $(i_m,j_m)_{m\in\N}\in Adm^\N$ that verifies two properties, explained in the following proposition.

\begin{proposition}\label{sequence}
There exists $c>0$ such that for all $m\in\N$, there exists $(i_m,j_m)\in Adm$ such that 
\begin{equation}\label{density_1}
j_m\notin \EE^{i_m}_m,\quad i.e.\ \frac{\HH^\alpha(A\cap I^{i_m}_{j_m})}{\HH^\alpha(I^{i_m}_{j_m})}\geq 1-\frac{1}{m},
\end{equation}
and
\begin{equation} \label{density_image}
h^{-i_ms\alpha}\HH^{\alpha}\big(g(A\cap I^{i_m}_{j_m})\big) \geq c.
\end{equation}
\end{proposition}

The property \eqref{density_1} will imply that the sequence of compact sets $\big(f_{i_m,j_m}^{-1}(A\cap I^{i_m}_{j_m})\big)_{m\in\N}$ converges to $[0,1]^N$ in the Hausdorff distance, whereas the property \eqref{density_image} will imply that $\HH^\alpha\big(f([0,1]^N)\big)>0$, where $f$ is the blow-up map. The proof of Proposition \ref{sequence} requires some lemmas.

The next lemma proves that in a small ball $B(x,r)$ where $x$ is a point of density of a set $E$ in $\R^N$, a "good" cover of $B(x,r)$ by cubes has a small number of cubes that have density in $E$ not close to $1$. 

\begin{lemma}\label{number_bad}
Let $x$ be a point of density for a subset $E$ of $\R^N$. For all $\varepsilon >0$ and all $m\in\N$, there exists $r_0>0$ such that for all $r\in (0,r_0)$, for all $i\in\N$ and all coverings $\{I^i_j\}_{j\in J^i}$ of $B(x,r)$ such that
$$\frac{\HH^\alpha\Big(\bigcup_{j\in J^i}I^i_j\Big)}{\HH^\alpha\big(B(x,r)\big)}\leq1+\frac{\varepsilon}{m},$$
then 
$$0\leq \frac{\#(J^i\cap\EE^i_m)}{\#J^i}\leq\varepsilon.$$
\end{lemma}

\begin{proof}
Let $\varepsilon >0, m\in\N$ and $\eta >0$. Since $x$ is a point of density for $E$, there exists $r_0>0$ such that for all $r\in(0,r_0)$,
$$1-\eta \leq \frac{\HH^\alpha\big(E\cap B(x,r)\big)}{\HH^\alpha\big(B(x,r)\big)}\leq 1.$$
Fix $r\in (0,r_0)$. Let $i\in\N$ and $\{I^i_j\}_{j\in J^i}$ be a covering of $B(x,r)$ such that 
\begin{equation}\label{covering}\frac{\HH^\alpha\Big(\bigcup_{j\in J^i}I^i_j\Big)}{\HH^\alpha\big(B(x,r)\big)}\leq1+\frac{\varepsilon}{m},
\end{equation}
then, 
\begin{align}
(1-\eta)\HH^\alpha\big(B(x,r)\big) &\leq \HH^\alpha\big(E\cap B(x,r)\big) \nonumber \\
						 &\leq \HH^\alpha\Big(E\cap\bigcup_{j\in J^i}I^i_j\Big)\nonumber \\
						 &\leq \sum_{j\in J^i} \HH^\alpha(E\cap I^i_j)\nonumber \\
						 &= \sum_{j\in J^i\cap\EE^i_m} \HH^\alpha(E\cap I^i_j) + \sum_{j\in J^i\setminus\EE^i_m} \HH^\alpha(E\cap I^i_j) \nonumber \\
						 &\leq (1-1/m)\sum_{j\in J^i\cap\EE^i_m}\HH^\alpha(I^i_j) +\sum_{j\in J^i\setminus \EE^i_m}\HH^\alpha(I^i_j) \nonumber.
\end{align}
For all $j\in J^i$, $\displaystyle\HH^\alpha(I^i_j)=\frac{1}{\# J^i}\HH^\alpha\Big(\bigcup_{j\in J^i}I^i_j\Big)$. Using $\# J^i=\#(J^i\cap\EE^i_m)+\#(J^i\setminus\EE^i_m)$ we get
\begin{align}
(1-\eta)\HH^\alpha\big(B(x,r)\big) &\leq \frac{1}{\# J^i}\HH^\alpha\Big(\bigcup_{j\in J^i}I^i_j\Big)\Big((1-1/m)\#(J^i\cap\EE^i_m)+\#(J^i\setminus\EE^i_m)\Big) \nonumber \\
&\leq \HH^\alpha\Big(\bigcup_{j\in J^i}I^i_j\Big)\left(1-\frac{\#(J^i\cap\EE^i_m)}{m\# J^i}\right)\nonumber.
\end{align}
Using now the inequality \eqref{covering}, we get
$$(1-\eta)\leq (1+\varepsilon/m)\left(1-\frac{\#(J^i\cap\EE^i_m)}{m\# J^i}\right)$$
which is the same as
$$\frac{\#(J^i\cap\EE^i_m)}{\# J^i}\leq m\left(1-\frac{1-\eta}{1+\varepsilon/m}\right).$$
Set $\eta=(\varepsilon/m)^2$ to get the lemma.
\end{proof}

With this lemma, it is easy to get local information on the sum of the measures of all the $E\cap I^i_j$ such that $I^i_j$ has density in $E$ not close to $1$.

\begin{coro}
Let $x$ be a point of density for a subset $E$ of $\R^N$. For all $\varepsilon >0$, and all $m\in\N$, there exists $r_0>0$ such that for all $r\in (0,r_0)$, for all $i\in\N$ and all coverings $\{I^i_j\}_{j\in J^i}$ of $B(x,r)$ such that
$$\frac{\HH^\alpha\Big(\bigcup_{j\in J^i}I^i_j\Big)}{\HH^\alpha\big(B(x,r)\big)}\leq1+\frac{\varepsilon}{m},$$
then 
$$\sum_{j\in J^i\cap\EE^i_m}\HH^\alpha(E\cap I^i_j) \leq \varepsilon\left(1+\frac{\varepsilon}{m}\right)\HH^\alpha\big(B(x,r)\big).$$
\end{coro}
\begin{proof}
Let $\varepsilon >0$, $m\in\N$. Let $r_0>0$ be given by Lemma \ref{number_bad}. For all $r\in (0,r_0)$, all $i\in\N$ and all coverings $\{I^i_j\}_{j\in J^i}$ such that 
\begin{equation*}
\frac{\HH^\alpha\Big(\bigcup_{j\in J^i}I^i_j\Big)}{\HH^\alpha\big(B(x,r)\big)}\leq1+\frac{\varepsilon}{m},
\end{equation*}
we have
\begin{align}
\sum_{j\in J^i\cap\EE^i_m}\HH^\alpha(E\cap I^i_j) &\leq \sum_{j\in J^i\cap\EE^i_m}\HH^\alpha(I^i_j) \nonumber \\
									  &\leq \frac{1}{\# J^i}\HH^\alpha\Big(\bigcup_{j\in J^i}I^i_j\Big) \#(J^i\cap\EE^i_m) \nonumber \\
									  &\leq \varepsilon \HH^\alpha\Big(\bigcup_{j\in J^i}I^i_j\Big) \quad \text{by Lemma \ref{number_bad}} \nonumber \\
								 	  &\leq \varepsilon\left(1+\frac{\varepsilon}{m}\right)\HH^\alpha\big(B(x,r)\big). \nonumber
\end{align}
\end{proof}

By a compactness argument, we can now deduce global information on the sum of the measures of the $E\cap I^i_j$ where $I^i_j$ is a cube with density in $E$ not close to $1$.

\begin{lemma}\label{measure_bad}
Suppose that $E$ is a compact subset of $\R^N$ such that $\HH^\alpha(E)>0$. Then for all $\varepsilon >0$, all $m\in\N$, there exists $i_0\in\N$ such that for all $i\geq i_0$, 
$$\sum_{j\in \EE^i_m}\HH^\alpha(E\cap I^i_j)\leq \varepsilon.$$
\end{lemma}

\begin{proof}

Let $K$ be a sufficiently large compact containing $E$. Let $\varepsilon >0$ and $m\in\N$. Denote by $E'$ the set of points of density in $E$. Recall that $\HH^\alpha(E\setminus E')=0$. By inner regularity of $\HH^\alpha$, there exists a compact set $E''\subset E'$ such that $\HH^\alpha(E'\setminus E'')\leq \varepsilon$. For all $x\in E''$, there exists $r_0(x)>0$ given by Lemma \ref{number_bad}. For all $x\in E''$, choose $r_x\in (0,r_0(x)/3)$ such that $B(x,r_x)\in K$. Then $\{B(x,r_x)\}_{x\in E''}$ covers $E''$. By compactness, there exists a finite subfamily $\{B(x_k,r_k)\}_{1\leq k\leq n}$ of $\{B(x,r_x)\}_{x\in E''}$ that covers $E''$. 
By Vitali covering Lemma, we can extract another subfamily of disjoint balls, say $\{B(x_k,r_k)\}_{1\leq k\leq n_0}$ such that 
$$E''\subset \bigcup_{k=1}^{n_0} B(x_k,3r_k).$$
Then choose $i_0$ sufficiently large such that for all $i\geq i_0$ and all $1\leq k\leq n_0$, 
$$\frac{\HH^\alpha\Big(\bigcup_{j\in J^i_k}I^i_j\Big)}{\HH^\alpha\big(B(x_k,3r_k)\big)}\leq1+\frac{\varepsilon}{m},$$
where $J^i_k=\left\{j\in\Z^N, B(x_k,3r_k)\cap I^i_j\neq\varnothing\right\}$.
Then 
\begin{align}
\sum_{j\in\EE^i_m}\HH^\alpha(E\cap I^i_j)&\leq \HH^\alpha(E\setminus E')+ \HH^\alpha(E'\setminus E'')+ \sum_{j\in\EE^i_m}\HH^\alpha (E''\cap I^i_j)\nonumber \\
								&\leq \varepsilon +\sum_{k=1}^{n_0}\sum_{j\in J^i_k\cap\EE^i_m} \HH^\alpha(E''\cap I^i_j) \nonumber \\
								&\leq \varepsilon +\sum_{k=1}^{n_0}\varepsilon\left(1+\frac{\varepsilon}{m}\right)\HH^\alpha\big(B(x_k,3r_k)\big) \quad\text{by Corollary }1 \nonumber \\
								&\leq \varepsilon +C\varepsilon\left(1+\frac{\varepsilon}{m}\right) \sum_{k=1}^{n_0} \HH^\alpha\big(B(x_k,r_k)\big) \quad\text{by Ahlfors regularity} \nonumber \\
								&\leq \varepsilon +C\varepsilon\left(1+\frac{\varepsilon}{m}\right)\HH^\alpha\Big(\bigcup_{k=1}^{n_0}B(x_k,r_k)\Big) \quad\text{since the balls are disjoint} \nonumber \\
								&\leq \varepsilon +C\varepsilon\left(1+\frac{\varepsilon}{m}\right) \HH^\alpha(K), \nonumber 
\end{align}
which proves the lemma.
\end{proof}

\begin{proof}[Proof of Proposition \ref{sequence}]
Suppose that Proposition \ref{sequence} is false. For all $c>0$, there exists $m\in\N$ such that for all $(i,j)\in Adm$, either $j\in\EE^i_m$ or $h^{-is\alpha}\HH^\alpha\big(g(A\cap I^{i}_j)\big) <c$. 

Fix $c>0$ and the corresponding integer $m$. By Lemma \ref{measure_bad}, there exists an increasing sequence $(i_n)_{n\in\N}$ of integers such that for all $n\in\N$,
$$\sum_{j\in\EE^{i_n}_m}\HH^\alpha(A\cap I^{i_n}_j)\leq\frac{1}{2^n}.$$
Given $n_0\in\N$, define $\FF_{n_0}=\varnothing$ and iteratively for $n>n_0$
$$\FF_n=\left\{j\in\Z^N|\ (i_n,j)\in Adm,\ h^{-i_ns\alpha}\HH^\alpha\big(g(A\cap I^{i_n}_j)\big)< c,\text{ and } I^{i_n}_j\cap\left(\bigcup_{n'=n_0}^{n-1}\bigcup_{j'\in\FF_{n'}}I^{i_{n'}}_{j'}\right)=\varnothing\right\}.$$
Now we set
$$B=A\setminus\bigcup_{n=n_0}^{+\infty}\Bigg(\bigcup_{\substack{j\in\EE^{i_n}_m\\ (i_n,j)\in Adm}}I^{i_n}_j\cup \bigcup_{j\in\FF_n}I^{i_n}_j\Bigg).$$

Let us prove that $\HH^\alpha(B)=0$. Suppose that $\HH^\alpha(B)>0$. Choose a point of density $x\in B$. We prove that there exists $\varepsilon >0$ and a sequence $(r_n)_{n\in\N}$ such that $r_n\to 0$ and for all $n\in\N$, 
\begin{equation}\label{estimate_ball}
\frac{\HH^\alpha\big(B\cap B(x,r_n)\big)}{\HH^\alpha\big(B(x,r_n)\big)}\leq 1-\varepsilon,
\end{equation}
which contradicts the fact that $x$ is a point of density. 
Let $r_n=2Nh^{(i_n-1)s}$. By definition of $r_n$, 
$$\prod_{k=1}^N[x_k-h^{(i_n-1)s/s_k},x_k+h^{(i_n-1)s/s_k}]\subset \overline{B(x,r_n/2)},$$
thus if we set $j_n=(\lfloor x_1/h^{(i_n-1)s/s_1}\rfloor,\dots,\lfloor x_N/h^{(i_n-1)s/s_N}\rfloor)$, then 
$$I^{i_n-1}_{j_n}\subset\prod_{k=1}^N[x_k-h^{(i_n-1)s/s_k},x_k+h^{(i_n-1)s/s_k}]\subset \overline{B(x,r_n/2)}.$$

For all $n\in\N$, there exists $j'_n\in \Z^N$ such that $(i_n,j'_n)\in Adm$ and $I^{i_n}_{j'_n}\subset I^{i_n-1}_{j_n}$. Now there are two cases: either $j'_n\in\EE^{i_n}_m$ or $h^{-i_ns\alpha}\HH^\alpha\big(g(A\cap I^{i_n}_{j'_n})\big)<c$.

\begin{figure}
\begin{center}
\begin{tikzpicture}[scale=1.7]
\draw (-2,0) node{$\bullet$};
\draw (-2,0) node[below left]{$x$};
\draw (-1,0) arc(0:120:1);
\draw (0,0) arc(0:120:2);
\draw (-2.3,0.4)--(-2.3,-0.4)--(-1.6,-0.4)--(-1.6,0.4)--(-2.3,0.4);
\draw (-1.95,0.3)--(-1.95,-0.1)--(-1.65,-0.1)--(-1.65,0.3)--(-1.95,0.3);
\draw (-1,1) node[above left]{$B(x,r_n/2)$};
\draw (0.5,1) node[above left]{$B(x,r_n)$};
\draw (-2.1,-0.77) node[above]{$I^{i_n-1}_{j_n}$};
\draw (-1.78,0.063) node{$I^{i_n}_{j'_n}$};
\draw [dotted] (-1.5,0.7)--(-1.5,0.2)--(-2.5,0.2)--(-2.5,0.7);
\draw (-2.4,0.65) node[left]{$I^{i_{n'}}_{j'}$};
\draw [dotted](-1.95,0.6)--(-1.95,0.2)--(-1.65,0.2)--(-1.65,0.6)--(-1.95,0.6);
\draw [very thin] (-1.8,0.25)--(-1.8,0.5);
\draw [very thin] (-1.75,0.45)--(-1.8,0.5);
\draw [very thin] (-1.85,0.45)--(-1.8,0.5);
\end{tikzpicture}
\caption{Construction of a cube isometric to $I^{i_n}_{j_n'}$ that does not intersect $B\cap B(x,r_n)$.}
\label{density}
\end{center}
\end{figure}
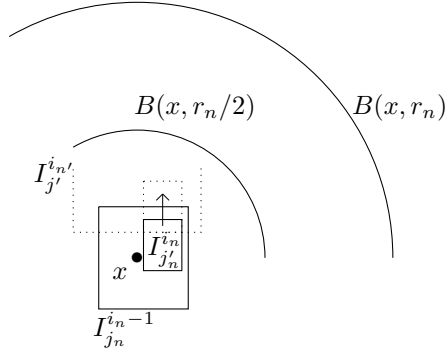

\begin{enumerate}
\item If $j'_n\in \EE^{i_n}_m$, then $B\cap I^{i_n}_{j'_n}=\varnothing$. Thus $B\cap B(x,r_n)=B\cap \big(B(x,r_n)\setminus I^{i_n}_{j'}\big)$, and we can write 
$$\frac{\HH^\alpha\big(B\cap B(x,r_n)\big)}{\HH^\alpha\big(B(x,r_n)\big)}\leq \frac{\HH^\alpha\big(B(x,r_n)\setminus I^{i_n}_{j'_n}\big)}{\HH^\alpha\big(B(x,r_n)\big)}\leq 1 -\frac{\HH^\alpha(I^{i_n}_{j'_n})}{Cr_n^\alpha}=1-h^{s\alpha}\frac{\HH^\alpha([0,1]^N)}{C(2N)^\alpha},$$
where $C$ is a constant given by Ahlfors regularity. This proves \eqref{estimate_ball} in this case. 
\item If $h^{-i_ns\alpha}\HH^\alpha\big(g(A\cap I^{i_n}_{j'_n})\big)<c$, then either $j'_n\in \FF_n$ or $j'_n\notin \FF_n$. If $j'_n\in \FF_n$, then $B\cap I^{i_n}_{j'_n}=\varnothing$ and we conclude as in the first case. If $j'_n\notin \FF_n$, then
$$I^{i_n}_{j'_n}\cap\left(\bigcup_{n'=n_0}^{n-1}\bigcup_{j'\in\FF_{n'}}I^{i_{n'}}_{j'}\right)\neq\varnothing.$$
Thus there exists $n'\in\left\{n_0,\dots, n-1\right\}$ and $j'\in\FF_{n'}$ such that $I^{i_n}_{j'_n}\cap I^{i_{n'}}_{j'}\neq \varnothing$. By construction, $B\cap I^{i_{n'}}_{j'}=\varnothing$. If one translates the cube $I^{i_n}_{j'_n}$ so that the image lies in $I^{i_{n'}}_{j'}$ and intersects $I^{i_n}_{j'_n}$, then the translated cube is in $B(x,r_n)$ but not in $B$, and has the same measure that $I^{i_n}_{j'_n}$, so the same estimate as in the first case holds (see Figure \ref{density}).
\end{enumerate}

Thus $\HH^\alpha(B)=0$, and since $g$ is Lipschitz, $\HH^\alpha\big(g(B)\big)=0$. 
With these we can write 
\begin{align*}
\HH^\alpha\big(g(A)\big) &\leq \HH^\alpha\big(g(A\setminus B)\big) +\HH^\alpha\big(g(A\cap B)\big) \\
				   &\leq \sum_{n=n_0}^{+\infty}\Bigg(\sum_{\substack{j\in\EE^{i_n}_m \\ (i_n,j)\in Adm}}\HH^\alpha\big(g(A\cap I^{i_n}_j)\big) +\sum_{j\in \FF_n}\HH^\alpha\big(g(A\cap I^{i_n}_{j})\big)\Bigg) \\
				   &\leq \sum_{n=n_0}^{+\infty}\Bigg(\sum_{\substack{j\in\EE^{i_n}_m \\ (i_n,j)\in Adm}}L^\alpha \HH^\alpha(A\cap I^{i_n}_j) +c\sum_{j\in\FF_n} h^{i_ns\alpha}\Bigg) \nonumber \\
				   &\leq \sum_{n=n_0}^{+\infty}\Bigg(L^\alpha \sum_{j\in\EE^{i_n}_m}\HH^\alpha(A\cap I^{i_n}_j) +c\sum_{j\in\FF_n} \frac{\HH^\alpha(I^{i_n}_j)}{\HH^\alpha([0,1]^N)}\Bigg)\\
				   &\leq L^\alpha\sum_{n=n_0}^{+\infty}\frac{1}{2^n}+\frac{c}{\HH^\alpha([0,1]^N)}\sum_{n=n_0}^{+\infty}\sum_{j\in\FF_n}\HH^\alpha(I^{i_n}_j) \\
				  &\leq L^\alpha \frac{1}{2^{n_0-1}}+c\frac{\HH^\alpha([0,1]^N)}{\HH^\alpha([0,1]^N)},
\end{align*}
the last inequality is true since the cubes $(I^{i_n}_j)_{n\geq n_0, j\in\FF_n}$ have disjoint interiors by construction. Since the last inequality is true for all $c>0$ and all $n_0\in\N$, then $\HH^\alpha\big(g(A)\big)=0$, which contradicts the assumption on $A$.\end{proof}

Until the end of this section, we fix a constant $c>0$ and a sequence $\{(i_m,j_m)\in Adm,\ m\in\N\}$ given by Proposition \ref{sequence}. For all $m\in\N$, set $f_m=f_{i_m,j_m}$, $K_m=f_m^{-1}(A\cap I^{i_m}_{j_m})$, $g_m=g_{i_m,j_m}$ and $z_m=z_{i_m,j_m}$ the chosen point in $A\cap I^{i_m}_{j_m}$. For all $m\in\N$, $K_m$ is a compact of $[0,1]^N$.

\begin{lemma} There exists a subsequence of $(K_m)_{m\in\N}$ that converges to $[0,1]^N$ in the Hausdorff distance. 
\end{lemma}

\begin{proof} By Proposition \ref{sequence},
\begin{equation}\label{convergenc}\HH^\alpha(K_m)=h^{-i_ms\alpha}\HH^\alpha(A\cap I^{i_m}_{j_m})=\HH^\alpha\big([0,1]^N\big)\frac{\HH^\alpha(A\cap I^{i_m}_{j_m})}{\HH^\alpha(I^{i_m}_{j_m})}\underset{m\to +\infty}{\longrightarrow}\HH^\alpha\big([0,1]^N\big).
\end{equation}
For all $m\in\N$, $K_m$ is a compact of $[0,1]^N$, so by the Blaschke Theorem, we can suppose that $K_m\overset{H}{\underset{m\to +\infty}{\longrightarrow}}K$. By using \eqref{convergenc} and Lemma \ref{semi_continuity}, $\HH^\alpha(K)=\HH^\alpha([0,1]^N)$. Then $K=[0,1]^N$, because a compact strictly contained in $[0,1]^N$ has Hausdorff $\alpha$-measure strictly less than $\HH^\alpha\big([0,1]^N\big)$. 
\end{proof}

We can now prove Theorem \ref{blow-up}.
\begin{proof}[Proof of Theorem \ref{blow-up}] The proof is pretty much the same as the proof of Ascoli-Arzel\`a Theorem.
Let $E$ be a dense subset of $[0,1]^N$. By Proposition \ref{kuratowski}, for all $x\in E$, there exists a sequence $(x_m)$ such that for all $m\in\N,\ x_m\in K_m$, and $d_\s(x_m,x)\to 0$. From now on, we fix such a sequence, for all $x\in E$. 

We now prove that there exists a compact $K$ such that for all $m\in\N, g_m(K_m)\subset K$. Let $m\in\N$ and $x\in K_m$. Then

\begin{align}
d_\s(g_m(x),0) 	&\leq h^{-i_ms}d_\s\big(g(f_m(x))-g(z_m),0\big) \nonumber \\
				&\leq h^{-i_ms}d_\s\big(g(f_m(x)),g(z_m)\big) \nonumber \\
				&\leq Lh^{-i_ms}d_{\s,\RR}(f_m(x),z_m) \nonumber \\
				&\leq Lh^{-i_ms}d_\s(f_m(x),z_m) \nonumber \\
				&\leq L h^{-i_ms}\text{diam}_{d_\s} I_{j_m}^{i_m},\nonumber
\end{align}
but $\text{diam}_{d_\s} I_{j_m}^{i_m}=Nh^{i_ms}$; then $d_\s\big(g_m(x),0\big)\leq L N.$ Thus $g_m(K_m)\subset \overline{B}(0,LN)$, which is compact. The usual Cantor diagonalization argument allows us to choose a subsequence (that we will still denote by $m$) such that $\big(g_{m}(x_{m})\big)_{m\in\N}$ converges (for the distance $d_\s$) for all $x\in E$ (that is, for all the sequences $(x_m)$ fixed above). Denote $f(x)$ the limit. The map $f:(E,d_{\s,\RR})\to (\R^N,d_\s)$ is Lipschitz. In fact, for all $x,y\in E$, let $(x_m), (y_m)$ the two fixed sequences converging to $x,y$ for the distance $d_\s$. Then for all $m\in\N$,
\begin{align*}
d_{\s,\RR}\big(f(x),f(y)\big) &\leq d_{\s,\RR}\big(f(x),g_m(x_m)\big) + d_{\s,\RR}\big(g_m(x_m),g_m(y_m)\big)+d_{\s,\RR}\big(g_m(y_m),f(y)\big) \\
					&\leq  d_{\s}\big(f(x),g_m(x_m)\big) + L d_{\s}(x_m,y_m)+d_{\s}\big(g_m(y_m),f(y)\big) \\
					&\leq  d_{\s}\big(f(x),g_m(x_m)\big) + L\Big( d_{\s}\big(x_m,x\big)+d_{\s}(x,y)+d_{\s}(y,y_m)\Big)\\
					&\quad +d_{\s}\big(g_m(y_m),f(y)\big)\\
					&\leq L d_\s(x,y).
\end{align*}
By a standard argument, there is a unique extension of $f$ (that we still denote by $f$) defined on $[0,1]^N$ such that $f:([0,1]^N,d_{\s,\RR})\to (\R^N,d_\s)$ is Lipschitz. Moreover, $(g_m)$ "converges pointwise" in the following sense: for all $x\in [0,1]^N$, all sequences $(x_m)$ (not only the ones fixed above) such that $x_m\in K_m$ and $d_{\s,\RR}(x_m,x)\underset{m\to +\infty}{\longrightarrow}0$, we have $d_\s\big(g_m(x_m),f(x)\big)\underset{m\to+\infty}{\longrightarrow}0$. In fact, by density of $E$, for all $p\in\N$, there exists $x^{(p)}\in E$ such that $d_\s(x^{(p)},x)\leq 1/2^p$. For all $p$, denote by $(x^{(p)}_m)_{m\in\N}$ the sequence converging to $x^{(p)}$ that we choose above. Then 
\begin{align*}
d_\s\big(g_m(x_m),f(x)\big) &\leq d_\s\big(g_m(x_m),g_m(x_m^{(p)})\big) + d_\s\big(g_m(x_m^{(p)}),f(x^{(p)})\big) + d_\s\big(f(x^{(p)}),f(x)\big) \\
					 &\leq L d_{\s,\RR}(x_m,x_m^{(p)}) + d_\s\big(g_m(x_m^{(p)}),f(x^{(p)})\big) + L d_{\s,\RR}(x^{(p)},x)
\end{align*}
and the result holds by letting $m\to +\infty$ and then $p\to +\infty$.

We now want to prove that $\HH^\alpha\big(f([0,1]^N)\big)>0$. By Lemma \ref{semi_continuity}, and Proposition \ref{sequence}, it is sufficient to prove that $d_H\big(g_m(K_m),f([0,1]^N)\big)\underset{m\to +\infty}{\longrightarrow}0$, because then 
$$0 < c\leq \underset{m\to +\infty}{\lim\sup}\ \HH^\alpha\big(g_m(K_m)\big) \leq \HH^\alpha\big(f([0,1]^N)\big).$$
Let us remark that since $g_m(K_m)\subset \overline{B}(0,Ln)$, then $f([0,1]^N)\subset \overline{B}(0,LN)$ by pointwise convergence, thus we can use the equivalence in Proposition \ref{kuratowski} to prove that $g_m(K_m)\overset{H}{\longrightarrow}f([0,1]^N)$.
The first point of Proposition \ref{kuratowski} is true by pointwise convergence of $g_m$ to $f$. For the second point, let $y=\underset{k\to +\infty}{\lim} g_{m_k}(x_{m_k})$ where, $\{x_m\}$ is a sequence such that for all $m\in\N, x_m\in K_m$. We want to prove that $y\in f([0,1]^N)$. Without loss of generality, since $K_m\subset[0,1]^N$ we can suppose that $x_{m_k}\underset{k\to +\infty}{\longrightarrow}x\in [0,1]^N$. Then $g_{m_k}(x_{m_k})\to f(x)=y$ by pointwise convergence. 

This proves that $d_H\big(g_m(K_m),f([0,1]^N)\big)\underset{m\to +\infty}{\longrightarrow}0$, and thus $\HH^\alpha\big(f([0,1]^N)\big)\geq c>0$.
\end{proof}

\section{$(\R^N,d_\s)$ is not minimal for looking down}\label{conclusion}
We know that $(\R^N,d_\s)$ looks down on $([0,1]^N/d_{\s,\RR},\overline{d}_{\s,\RR})$, since $d_{\s,\RR}\leq d_\s$ and the projection $\pi : ([0,1]^N,d_{\s,\RR})\to ([0,1]^N/d_{\s,\RR},\overline{d}_{\s,\RR})$ is an isometry. 
\begin{proposition}\label{proposition_lipschitz} A Lipschitz map $f:([0,1]^N,d_{\s,\RR})\to (\R^N,d_\s)$ verifies 
$\HH^\alpha\big(f([0,1]^N)\big)=0$.
\end{proposition}

\begin{proof} Let $x,y\in\R^N$ and $i\in L$. Let $\pi_i$ be the canonical projection on the $i$-th coordinate. We denote by $\gamma :[1,N]\to \R^N$ the following piecewise linear curve: if $x=(x_1,\dots,x_N)$ and $y=(y_1,\dots,y_N)$ then for all $k\in\left\{1,\dots,N\right\}$
$$\gamma |_{[k,k+1]}(t)=(y_1,\dots,y_{k-1},x_k+(t-k)(y_k-x_k),x_{k+1},\dots,x_n).$$
$\gamma$ is a geodesic between $x$ and $y$ with the distance induced by the classical norm $\|\cdot\|_1$ in $\R^N$. For all $k\in\left\{1,\dots,N\right\}$ the map
$$\begin{array}{cccc}
\varphi_k : &([0,1],d_k) &\longrightarrow& (\R,|\cdot|^{s}) \\
&t&\longmapsto&(\pi_i\circ f)\big(\gamma(k+t)\big)
\end{array}
$$ 
is a Lipschitz map since $\gamma|_{[k,k+1]}$ and $\pi_i$ are Lipschitz. Here, $d_k=|\cdot|^{s_k}$ if $k\notin L$, and $d_k=d_{\RR}$ if $k\in L$. The case $k\notin L$ implies by a standard argument that $\varphi_k$ is constant, since $s<s_k$. The case $k\in L$ also implies that $\varphi_k$ is constant by Proposition \ref{no_lipschitz_function}. We conclude that $\pi_i\circ f$ is constant along the curve $\gamma$ and thus $\pi_i\big(f(x)\big)=\pi_i\big(f(y)\big)$, which proves that $\HH^\alpha\big(f([0,1]^N)\big)=0$.
\end{proof}
Suppose that $([0,1]^N/d_{\s,\RR},\overline{d}_{\s,\RR})$ looks down on $(\R^N,d_\s)$. Then there exist a closed subset $A\subset [0,1]^N/d_{\s,\RR}$ and an $L$-Lipschitz map $g:(A,\overline{d}_{\s,\RR})\to (\R^N,d_\s)$ such that $\HH^\alpha_{d_\s}(g(A))>0$. The map $\tilde{g}=g\circ \pi:(\pi^{-1}(A),d_{\s,\RR})\to (\R^N,d_\s)$ is Lipschitz, $\HH^\alpha_{d_\s}\big(\tilde{g}(A)\big)>0$ and $\pi^{-1}(A)$ is compact. Thus we can apply Theorem \ref{blow-up}. There exists a Lipschitz map $f:([0,1]^N,d_{\s,\RR})\to (\R^N,d_\s)$ such that $\HH^\alpha\big(f([0,1]^N)\big)>0$, which contradicts Proposition \ref{proposition_lipschitz}.

\nocite{*}
\bibliographystyle{amsplain}
\bibliography{bibli}

\end{document}